\documentclass[journal,twoside,web]{IEEEcolor}
\usepackage{generic}
\usepackage{stfloats}
\usepackage{savesym}
\usepackage{amsmath}
\usepackage{textcomp}
\savesymbol{iint}
\restoresymbol{TXF}{iint}
\usepackage{graphicx,subfigure}
\usepackage{epstopdf}
\usepackage{amssymb,amsfonts}
\usepackage{cite}
\usepackage{latexsym}
\usepackage{psfrag}
\usepackage{pifont}
\usepackage{caption}
\usepackage{color}
\usepackage{tikz}
\usepackage{bm}
\usetikzlibrary{arrows,shapes,chains}
\usepackage{mathtools}
\usepackage{cuted}
\usepackage{setspace}
\usepackage{algorithm}
\usepackage{algpseudocode}
\usepackage{amsmath}
\usepackage{mathrsfs}
\allowdisplaybreaks[1]

\let\labelindent\relax
\usepackage{enumitem}
\newtheorem{theorem}{Theorem}   
\newtheorem{lemma}{Lemma}
\newtheorem{problem}{Problem}
\newtheorem{proposition}{Proposition}

\newtheorem{remark}{Remark}

\newtheorem{definition}{Definition}

\DeclareMathOperator{\rank}{rank}

\DeclareMathOperator{\im}{im}
\DeclareMathOperator{\tr}{tr}

\begin{document}

\title{Bridging Model Reference Adaptive\\ Control and Data Informativity}
	
\author{Jiwei Wang, Simone Baldi,~\IEEEmembership{Senior Member,~IEEE} and Henk J. van Waarde,~\IEEEmembership{Member,~IEEE} 
		\thanks{This work was partially supported by Jiangsu Provincial Scientific Research Center of Applied Mathematics grant BK20233002.
		H. J. van Waarde acknowledges financial support by the Dutch Research Council under the NWO Talent Programme Veni Agreement (VI.Veni.222.335). (Corresponding author: Simone Baldi)}
		\thanks{J. Wang is with School of Cyber Science and Engineering, Southeast University, China, and with the Bernoulli Institute for Mathematics, Computer Science and Artificial Intelligence, University of Groningen, The Netherlands (email: jiwei.wang@rug.nl).}
		\thanks{S. Baldi is with School of Mathematics, Southeast University, China (email: simonebaldi@seu.edu.cn).}
		\thanks{H. J. van Waarde is with the Bernoulli Institute for Mathematics, Computer Science and Artificial Intelligence, University of Groningen, The Netherlands (email: h.j.van.waarde@rug.nl).}}
	
\maketitle
	
\begin{abstract}

The goal of model reference adaptive control (MRAC) is to ensure that the trajectories of an unknown dynamical system track those of a given reference model. 
This is done by means of a feedback controller that adaptively changes its gains using data collected online from the closed-loop system. 
One of the approaches to solve the MRAC problem is to impose conditions on the data that guarantee convergence of the gains to a solution of the so-called matching equations. 
In the literature, various extensions of the concept of persistent excitation have been proposed in an effort to weaken the conditions on the data required for this convergence.
Despite these efforts, it is not well-understood what conditions are necessary and sufficient for ensuring convergence of MRAC to a solution of the matching equations.
In this paper, we propose a new framework to study the MRAC problem, using the concept of data informativity.
Our main contribution is to provide \emph{necessary and sufficient} conditions for the existence of an adaptive law that guarantees convergence of the gains to a solution of the matching equations, and to provide a recipe for its construction.
While existing excitation conditions imply that the system can be uniquely identified from the collected data, our results show that this is not necessary for the convergence of the feedback gains.
\end{abstract}
	
	
\section{Introduction}
	
Adaptive control has a rich history, which can be traced back to pioneering work from the 1950s \cite{aseltine1958survey}. 
Adaptive control aims at dealing with the inevitable uncertainty in control systems by means of adjustable-gain controllers, that are adapted online using the data generated by the closed-loop system.
Among the various adaptive control strategies developed over the years, a common approach is model reference adaptive control (MRAC), where the control gains are adjusted so that the closed-loop system tracks a predefined reference model  \cite{nguyen2018model}. 

Meanwhile, offline data-driven control approaches also have a rich history: here, the system to be controlled is uncertain as in adaptive control, but the data are collected offline \cite{willems2012data,de2019formulas}.
In this setting, the controller is constructed in one shot using a batch of data.
Among the various offline data-driven control strategies, the data informativity approach has recently emerged \cite{van2020data}. 
Data informativity provides necessary and sufficient conditions under which the data contain sufficient information for control design.
In what follows, we provide an overview of representative MRAC approaches and data informativity approaches, so as to clarify the gaps in the respective fields and how these gaps are closed in this work.


\subsection{State of the art on data requirements in MRAC}

Despite decades of history, MRAC still remains a prominent research area, with advances continually emerging, such as switched MRAC \cite{di2012hybrid,kersting2017direct,yuan2020lyapunov}, robust MRAC \cite{xie2018h,franco2021robust}, distributed MRAC \cite{guo2020distributed,yue2023model}, partial-state feedback MRAC \cite{song2019partial}, and learning-based MRAC \cite{wang2023mrac,annaswamy2023integration}, just to mention a few. 
The key concept common to all MRAC approaches are the matching equations: control gains must exist so that the closed-loop system matches the dynamics of the reference model. 
Assuming the matching equations have a solution, a pertinent question is whether the online adjustment mechanism converges to such a solution.
One of the approaches to MRAC is to impose conditions on the closed-loop data that guarantee convergence of the adjustment mechanism to a solution of the matching equations.

It is well known that the most traditional requirement on the closed-loop data in classical MRAC approaches (direct or indirect \cite{ioannou2006adaptive,duarte1989combined}) goes under the name of persistence of excitation (PE), which requires the excitation provided by the input data to never vanish \cite{boyd1986necessary, annaswamy2021historical}.
In the absence of PE, these classical adaptive schemes can still achieve control objectives like convergence of the tracking error \cite{aastrom1973self}, but the adaptive gains may fail to converge to the desired values. In other words, convergence of the gains to a desired solution imposes stronger requirements on data as compared to mere convergence of the tracking error.
Motivated by the restrictive nature of the PE, a productive line of research has been devoted to relaxing the requirements imposed on data that guarantee convergence to a solution of the matching equations.
Representative methods are composite/combined MRAC \cite{lavretsky2009combined}, concurrent learning MRAC \cite{chowdhary2010concurrent}, and works where PE has been relaxed to rank excitation \cite{chowdhary2013concurrent}, finite excitation \cite{cho2018composite,lee2019concurrent} and initial excitation \cite{roy2018combined,katiyar2022initial}.
While all these relaxed excitation conditions are sufficient to guarantee convergence to a solution of the matching equations, a natural question is whether they are also necessary.
Our work will show that such conditions are, in general, not necessary to guarantee convergence to a solution of the matching equations.
It is worth mentioning that, while the absence of PE conditions in some adaptive control schemes has been investigated in \cite{kogan1994locally,kogan1996locally}, these results do not apply to the MRAC problem (see the discussion in Section~\ref{excitation}).

Although most of the methods above address multi-variable MRAC problems (where the number of inputs and reference signals is larger than one \cite{tao2014multivariable}), such a multi-variable setting requires additional assumptions on the system to arrive at a solution to the problem.
It is common in multi-variable MRAC to require a known system's input matrix \cite{tao2014multivariable,lavretsky2009combined,chowdhary2010concurrent,chowdhary2013concurrent,cho2018composite,lee2019concurrent} or other structural assumptions \cite{costa2003lyapunov,imai2004multivariable,roy2018combined,song2021partial} to handle the uncertainty.
Our work will remove such restrictions, requiring no prior knowledge on the input matrix and no other structural assumptions on the unknown system apart from controllability.
Moreover, we are able to guarantee convergence to a solution of the matching equations even when the matching equations have a possibly infinite number of solutions, in contrast with existing multi-variable MRAC approaches where the structural assumptions are such that only one solution of the matching equations exists. 


\subsection{State of the art on adaptation in data informativity}

In the literature on offline data-driven control, a key role is played by Willems et al.'s fundamental lemma \cite{willems2005note}.
This result provides a data-dependent representation of the system based on persistently exciting data \cite{de2019formulas,breschi2021direct,schmitz2022willems,pan2022stochastic}.
Similar to the quest for relaxed PE requirements in adaptive control, research efforts have been made to relax the requirement for persistently exciting data. 
An answer in this direction was provided by the data informativity framework \cite{van2020data}, demonstrating that PE is not necessary in several data-driven analysis and control problems; also see \cite{van2023informativity} for more details.
By adopting the data informativity perspective, \cite{wang2025necessary} studied the requirements that data must have in order to solve an offline model reference control: there, it was shown that persistently exciting data are also not necessary for solving such problem. 

However, offline data-driven control is intrinsically different from MRAC, because the data are not generated online by the time-varying closed-loop system obtained from interconnecting the dynamical system and the adaptive feedback law.
Indeed, in existing work on the data informativity framework, controllers are constructed offline in one shot, using a batch of data.
Online experiment design procedures have been studied in \cite{van2021beyond}: however, such procedures only aim at generating a batch of informative data and do not incorporate a control objective, as is customary in adaptive control.
This motivates us to further investigate the relation between MRAC and data informativity.
In particular, a key question is: can data informativity be turned from an offline framework to an online one, where data are generated in real-time by the closed-loop system as the controller is adjusted online?
In this work, we develop a novel MRAC approach based on the concept of data informativity, thus bridging these two realms. 
Because data informativity has been shown to give necessary and sufficient conditions for several data-driven control problems, a crucial question is: 
can necessary and sufficient conditions be established for the existence of an adaptive law that ensures the convergence of MRAC to a solution of the matching equations?
This work gives a positive answer and provides a recipe to design such an adaptive law.

\subsection{Contributions}

This paper studies multi-variable MRAC using the framework of data informativity.
The main contributions are as follows:
\begin{itemize}
	\item 
	We establish a connection between data informativity for model reference control \cite{wang2025necessary} and data informativity for system identification (Proposition \ref{TMC} \& Theorem \ref{TBm}).
	Specifically, we show that data informativity for system identification is sufficient but not necessary for data informativity for model reference control.
	It becomes necessary only in special cases.
	This sets the stage for studying MRAC through the lens of data informativity.
	\item 
	We provide two equivalent conditions for the existence of an adaptive law guaranteeing asymptotic convergence of the associated gains to a solution of the matching equations (Theorem \ref{Texist}). 
	The first condition involves the solvability of the matching equations themselves. The second condition asserts that the underlying system is capable of generating data that are informative for model reference control.
	\item
	Following the existence analysis, we propose an MRAC approach that only exploits data informativity for model reference control, without relying on data informativity for system identification or persistence of excitation conditions.
	The framework relies on data informativity as in work by the authors on offline model reference control \cite{wang2025necessary}. 
	However, we develop new conditions for data informativity for model reference control that can be checked online via linear algebra, as data are generated at every time step by the closed-loop system.
	This provides an online method for collecting informative data (Lemma \ref{LT}), which is then utilized to design a controller along with an adaptive law. 
	The online collection of informative data marks a key difference with various online approaches \cite{rotulo2022online,liu2023online,eising2025data} where the conditions on data are assumed a priori rather than imposed by suitable design of the \mbox{inputs}.
	We also prove that, with the proposed online method, the controller gains will converge to a solution of the matching equations (Theorem \ref{TMRAC}), and the tracking error is bounded and will converge to zero (Theorem \ref{Pe}).
	Notably, our MRAC method operates on general multi-variable input-state systems without structural assumptions on the unknown system matrices.
	Even in scenarios where the ideal controller is not unique, our framework guarantees that the controller gains will converge to a solution of the matching equations.
	\end{itemize}

\subsection{Outline}

	The remainder of this paper is structured as follows.
	In Section II, notation and preliminary results about Willems et al.'s fundamental lemma are provided. 
	Section III presents the mathematical formulation of MRAC.
	In Section IV, key results on data informativity for system identification and data informativity for model reference control are recalled and compared.
	The proposed framework for online MRAC is presented in Section V: the framework includes necessary and sufficient conditions for the existence of an adaptive law that guarantees convergence of the gains to a solution of the matching equations, and a recipe for constructing such an adaptive law.
	This section also discusses how the proposed solution relates to the requirements imposed on data in state-of-the-art MRAC methods.
	In Section VI, numerical examples are provided.
	Concluding remarks are given in Section VII.
	
	
\section{Preliminaries}
	
\subsection{Notation}
	We denote the set of \emph{positive integers} by $ \mathbb{N} $, the set of \emph{non-negative integers} by $ \mathbb{Z}_+ $, and the set of \emph{non-negative real numbers} by $ \mathbb{R}_{+} $.
	The \textit{Euclidean norm} of a vector $ v \in \mathbb{R}^n $ is denoted by $ |v| $.
	The \textit{Frobenius norm} and \textit{induced $ 2 $-norm} of a matrix $ A\in \mathbb{R}^{n\times m} $ are denoted by $ \|A\|_F $ and $ \|A\|_2 $, respectively.
	The \textit{identity matrix} of appropriate size is denoted by $ I $.
	A square matrix is called \textit{Schur} if all its eigenvalues have modulus strictly less than 1, and is \textit{Hurwitz} if all its eigenvalues have negative real part.
	A symmetric matrix \mbox{$ A=A^\top \in \mathbb{R}^{n\times n} $} is said to be \textit{positive definite} (denoted by $ A>0 $) if $ v^{\top}Av>0 $ for all nonzero $ v\in\mathbb{R}^n $.
	Moreover, $ A $ is called \emph{positive semidefinite} if $v^\top A v \geq 0$ for all $v \in \mathbb{R}^n$. 
	This is denoted by $A \geq 0$. 
	For two symmetric $n \times n$ matrices $A$ and $B$, we use the notation $A \geq B$, meaning that $A-B \geq 0$. 
	We denote by $\ell_\infty^{n,m}$ the space of all sequences $\{S(t)\}_{t = 0}^\infty$ for which $S(t) \in \mathbb{R}^{n \times m}$ and
	$$
	|| S ||_\infty := \sup_{t \in \mathbb{Z}_+} ||S(t)||_2 < \infty. 
	$$
In addition, we use the shorthand notation $\ell_\infty^n := \ell_\infty^{n,1}$.
	

\subsection{Behaviors and Willems et al.'s fundamental lemma}

In this section, we recall the fundamental lemma by Willems et al. \cite{willems2005note}.
Consider a discrete-time input-state system
\begin{equation}\label{s}
x(t+1)=A_{\rm s}x(t)+B_{\rm s}u(t),
\end{equation}
where $t \in \mathbb{Z}_+$, $ x(t) \in \mathbb{R}^n $ is the state, $ u(t) \in \mathbb{R}^m $ is the control input, and $ A_{\rm s} \in \mathbb{R}^{n\times n} $ and $ B_{\rm s} \in \mathbb{R}^{n\times m} $ are the system matrices.
We assume throughout the paper that $ (A_{\rm s},B_{\rm s}) $ is controllable. 
Define the behavior 
$$
\mathfrak{B} := \{(u,x) : \mathbb{Z}_+ \to \mathbb{R}^m \times \mathbb{R}^n \mid \eqref{s} \text{ holds for all } t \in \mathbb{Z}_+\}.
$$
We call elements of $\mathfrak{B}$ (input-state) trajectories of \eqref{s}. 
Given a trajectory $(u,x) \in \mathfrak{B}$ and an integer $t \in \mathbb{N}$, define the matrices $U_-(t)$ and $X(t)$ as
\begin{equation}\label{d}
\begin{aligned}
U_-(t):=&\begin{bmatrix}
u(0)&u(1)&\cdots&u(t-1)
\end{bmatrix},
\\
X(t):=&\begin{bmatrix}
x(0)&x(1)&\cdots&x(t)
\end{bmatrix}.
\end{aligned}
\end{equation}
Moreover, we introduce the following matrices
\begin{equation}\label{sd}
\begin{aligned}
X_-(t):=&\begin{bmatrix}
x(0)&x(1)&\cdots&x(t-1)
\end{bmatrix},\\
X_+(t):=&\begin{bmatrix}
x(1)&x(2)&\cdots&x(t)
\end{bmatrix}.
\end{aligned}
\end{equation}
Then, for any $t \in \mathbb{N}$, the restricted behavior\footnote{Note that the notion of restricted behavior adopted here is slightly different from the literature \cite{willems2005note} because we work with input trajectories of length $t$ and state trajectories of length $t+1$. This is done to facilitate the ensuing discussion.} is defined as
$$
\mathfrak{B}_t := \{(U_-(t),X(t)) \mid (u,x) \in \mathfrak{B}\}.
$$
We call elements of $\mathfrak{B}_t$ restricted trajectories of \eqref{s}. 
In the subsequent analysis, we also view elements of $\mathfrak{B}_t$ as data sets. 
Let $ \ell $ be a positive integer such that $ \ell\leq t $.
Then, the \textit{Hankel matrix} of $ U_-(t) $ of depth $ \ell $ is denoted by
\begin{equation*}
\mathcal{H}_\ell(U_-(t))=
\begin{bmatrix}
u(0)&u(1)&\cdots&u(t-\ell)\\
u(1)&u(2)&\cdots&u(t-\ell+1)
\\
\vdots&\vdots&&\vdots\\
u(\ell-1)&	u(\ell)&\cdots&u(t-1)\\
\end{bmatrix}.
\end{equation*}
\begin{definition}
	$ U_-(t) $ is said to be \textit{persistently exciting of order $ \ell $} if $ \mathcal{H}_\ell(U_-(t)) $ has full row rank.
\end{definition}
By applying Willems et al.'s fundamental lemma to input-state systems, we obtain a rank condition on the input-state data, assuming that the input is persistently exciting of sufficiently high order.
\begin{lemma} \label{LW}
	($ \hspace{-0.96ex} $\cite[Corollary 2]{willems2005note})
	Let $t \in \mathbb{N}$ and consider the data $(U_-(t),X(t)) \in \mathfrak{B}_t$.
	If $ (A_{{\rm s}},B_{{\rm s}}) $ is controllable and $ U_-(t) $ is persistently exciting of order $ n + 1 $, then 
	\begin{equation}\label{frc}
	\rank\left(\begin{bmatrix}
	X_-(t)\\U_-(t)
	\end{bmatrix}\right)=n+m.
	\end{equation}
\end{lemma}


\section{Problem formulation}\label{SPF}

Consider the discrete-time reference system model
\begin{equation}\label{r}
x_{\rm m}(t+1)=A_{\rm m}x_{\rm m}(t)+B_{\rm m}r(t),
\end{equation}
where $ x_{\rm m}(t) \in \mathbb{R}^n $ is the reference state, $ r(t)\in \mathbb{R}^p $ is the reference input with $ p\leq m $, and $ A_{\rm m} \in \mathbb{R}^{n\times n}, B_{\rm m} \in \mathbb{R}^{n\times p} $ are the reference system matrices.
We assume throughout the paper that $ (A_{\rm m},B_{\rm m}) $ is controllable and $ A_{\rm m} $ is Schur.
The problem of model reference control is to find matrices $ K\in \mathbb{R}^{m\times n} $ and $ L\in \mathbb{R}^{m\times p} $ that satisfy the so-called \textit{matching equations} \cite{ioannou2006adaptive}
\begin{equation}\label{mc}
A_{\rm s}+B_{\rm s}K=A_{\rm m}, \quad
B_{\rm s}L=B_{\rm m}.
\end{equation}
We say that the model reference control problem is solvable if there exist matrices $K$ and $L$ satisfying \eqref{mc}.  
By \eqref{mc}, the closed-loop system formed by \eqref{s} and the fixed-gain controller 
\begin{equation*}
u(t) = Kx(t)+Lr(t),
\end{equation*}
results in the error dynamics $ e(t + 1) = A_{\rm m}e(t) $, where \mbox{$e(t)=x(t)-x_{\rm m}(t)$}. 
Since $A_{\rm m}$ is Schur, we note that $e(t)$ converges to zero for any reference input and any pair of initial states of the system \eqref{s} and of the reference model \eqref{r}.	
The model reference control problem becomes an MRAC problem when $ (A_{{\rm s}},B_{{\rm s}}) $ is unknown, and the above fixed-gain controller is replaced with an adaptive-gain controller
\begin{equation}\label{controller}
u(t) = \hat{K}(t)x(t)+\hat{L}(t)r(t)+v(t),
\end{equation} 
where the gains $ \hat{K}(t)\in \mathbb{R}^{m\times n} $ and $ \hat{L}(t)\in \mathbb{R}^{m\times p} $ are updated online at each time step using the input-state data, and the signal $ v:\mathbb{Z}_+\to\mathbb{R}^{\rm m} $ can be used to excite the system. 
In this work, we will consider $v$ having finite support, meaning that excitation can only be provided for a finite number of time steps.
Formally, the MRAC problem studied in this work is as follows.

\begin{problem}\label{P}
	Given the reference model \eqref{r}, provide necessary and sufficient conditions under which there exist, for \mbox{$ t\in\mathbb{Z}_+ $,}
	\begin{itemize}
		\item integers $ i_t,j_t\in\mathbb{N} $,
		\item functions $ \alpha_t $ and $ \beta_t $, and
		\item a finitely supported signal $ v $
	\end{itemize}
	such that, for any reference input $r \in \ell_\infty^p$, and any initial conditions 
	$$ x(0)\in\mathbb{R}^{n}, \hat{K}(0)\in \mathbb{R}^{m\times n}, \hat{L}(0)\in \mathbb{R}^{m\times p} \text{ and } \Theta(0)\in \mathbb{R}^{i_0\times j_0}, $$ 
	the adaptive law
	\begin{equation}
	\label{pTheta}
	\Theta(t+1) = \alpha_t(U_{-}(t),X(t),\Theta(t))\in \mathbb{R}^{i_{t+1} \times j_{t+1}}
	\end{equation}
	and the adaptive gains
	\begin{equation}
	\label{hatKL}
	\begin{bmatrix}
	\hat{K}(t)&\hat{L}(t)
	\end{bmatrix} = \beta_t(U_{-}(t),X(t),\Theta(t))
	\in \mathbb{R}^{m \times n} \times \mathbb{R}^{m \times p}
	\end{equation}
	are such that
\begin{equation}\label{converge}
	\lim_{t\to\infty} A_{\rm s}+B_{\rm s}\hat{K}(t) = A_{\rm m},\
	\lim_{t\to\infty} B_{\rm s}\hat{L}(t) = B_{\rm m}.
\end{equation}
Here, $ (U_{-}(t),X(t))\in\mathfrak{B}_t $ is the restricted trajectory of \eqref{s} resulting from $ x(0) $ and the control input in \eqref{controller}.
Moreover, if they exist, provide a recipe to construct $ i_t,j_t,\alpha_t,\beta_t $ and $ v $.
\end{problem}

\section {Data Informativity for Model Reference Control}

Before embarking on	the solution to the MRAC problem, we first solve a preliminary problem in the current section. 
Namely, we recall conditions under which a solution $ (K,L) $ of the matching equations \eqref{mc} can be obtained from data $ (U_-(t),X(t)) $. 
We will also compare these conditions to the ones required for unique system identification.

\subsection{Conditions on data for system identification and for model reference control}

	Consider the set of all systems compatible with the data $(U_-(t),X(t)) \in \mathfrak{B}_t$, given by
	\begin{align*}
	\Sigma_{(U_-(t),X(t))}=\{(A,B)\mid X_+(t) = AX_-(t) + BU_-(t)\}.
	\end{align*}
	Since the data $ (U_-(t),X(t)) $ are generated by system \eqref{s}, we obviously have $ (A_{\rm s},B_{\rm s})\in\Sigma_{(U_-(t),X(t))} $.
	However, the set $ \Sigma_{(U_-(t),X(t))} $ may contain other systems compatible with the data.
	This leads to the following definition.
	\begin{definition}
		($ \!\!  $\cite{van2020data}). 
		The data $ (U_-(t),X(t)) $ are \emph{informative for system identification} if $ \Sigma_{(U_-(t),X(t))} = \{(A_{\rm s},B_{\rm s})\} $.
	\end{definition}

	The following basic lemma provides necessary and sufficient conditions for system identification.
	\begin{lemma}\label{LDISI}
		($ \! \! $\cite[Proposition 6]{van2020data}).
		The data $ (U_-(t),X(t)) $ are informative for system identification if and only if \eqref{frc} holds.
	\end{lemma}
	
	Let us now concentrate on the problem of model reference control. 
	For given gains $ K\in \mathbb{R}^{m\times n} $ and $ L\in \mathbb{R}^{m\times p} $, we define the set of systems that match the reference model \eqref{r} as
	\begin{equation*}
	\Sigma^{K,L}=\{(A,B)\mid A+BK=A_{\rm m}, BL=B_{\rm m}\}.
	\end{equation*}
	Rather than system identification, the objective here is to find a solution $ (K,L) $ of the matching equations \eqref{mc}. 
	However, since $ (A_{\rm s},B_{\rm s}) $ is unknown and the data $ (U_-(t),X(t)) $ may not allow to distinguish the actual $ (A_{\rm s},B_{\rm s}) $ from any other system in $ \Sigma_{(U_-(t),X(t))} $, one needs to find control gains $ (K,L) $ such that \emph{all} systems compatible with the data satisfy the matching equations \eqref{mc}. 
	This leads to the following definition.
	
	\begin{definition}\label{D1} ($ \hspace{-0.96mm} $\cite{wang2025necessary}\hspace{-0.16mm}).
		The data $ (U_-(t),X(t)) $ are \emph{informative for model reference control} if there exist control gains $ (K,L) $ such that $ \Sigma_{(U_-(t),X(t))} \subseteq \Sigma^{K,L} $.
	\end{definition}
	
	The following lemma establishes necessary and sufficient conditions for the data $ (U_-(t),X(t)) $ to be informative for model reference control.
	\begin{lemma}\label{LDIMRC}
		The data $ (U_-(t),X(t)) $ are informative for model reference control if and only if \begin{equation}\label{image}
		\im\begin{bmatrix}
		I&0\\A_{\rm m}&B_{\rm m}
		\end{bmatrix}\subseteq
		\im\begin{bmatrix}
		X_-(t) \\ X_+(t)
		\end{bmatrix}.
		\end{equation}
	\end{lemma}
\vspace{0.6em}
\begin{proof}
	We prove that Lemma \ref{LDIMRC} is a reformulation of \cite[Theorem 1]{wang2025necessary}.
To show this, note that \eqref{image} is equivalent to the existence of matrices $ V_1\in\mathbb{R}^{t\times n} $ and $ V_2\in\mathbb{R}^{t\times p} $ such that
\begin{equation}\label{V12}
\begin{bmatrix}
I&0\\A_{\rm m}&B_{\rm m}
\end{bmatrix}
=\begin{bmatrix}
X_-(t) \\ X_+(t)
\end{bmatrix}
\begin{bmatrix}
V_1&V_2
\end{bmatrix},
\end{equation}
or, equivalently, 
\begin{equation}\label{mrc}
X_-(t)V_1\!=\!I,\ X_-(t)V_2\!=\!0,\ X_+(t)V_1\!=\!A_{\rm m},\ X_+(t)V_2\!=\!B_{\rm m},
\end{equation}
which are precisely the conditions of \cite[Theorem 1]{wang2025necessary}.
\end{proof}

We note that the conditions of Lemma \ref{LDIMRC} may hold even when the data do \textit{not} enable the unique identification of $ (A_{\rm s},B_{\rm s}) $, as shown in \cite[Example 1]{wang2025necessary}.
This implies that the full rank condition \eqref{frc} is not always necessary to solve the model reference control problem. 
In the following, we will show that, if a solution of the matching equations exists, data informativity for system identification is sufficient for data informativity for model reference control but necessary only in some special cases.
	
\subsection{Relations between the two data informativity conditions}
	\label{relation}
	
We start by showing that data informativity for system identification implies data informativity for model reference control if the matching equations \eqref{mc} have a solution.
	
\begin{proposition}\label{TMC}
	Assume that there exist gain matrices \mbox{$K \in \mathbb{R}^{m \times n}$} and $L \in \mathbb{R}^{m \times p}$ such that \eqref{mc} holds.
	Then, the data $ (U_-(t),X(t)) $ are informative for model reference control if $ (U_-(t),X(t)) $ are informative for system identification.
\end{proposition}

\begin{proof}
	Suppose that $ (U_-(t),X(t)) $ are informative for system identification.
	Let $ (K,L) $ be a solution of the matching equations \eqref{mc}.
	Then, by \eqref{frc}, there exist matrices $ V_1 \in \mathbb{R}^{t\times n} $ and $ V_2 \in \mathbb{R}^{t\times p} $ such that
	\begin{equation*}
	\begin{bmatrix}
	X_-(t)\\U_-(t)
	\end{bmatrix}V_1=
	\begin{bmatrix}
	I\\K
	\end{bmatrix},\
	\begin{bmatrix}
	X_-(t)\\U_-(t)
	\end{bmatrix}V_2=
	\begin{bmatrix}
	0\\L
	\end{bmatrix}.
	\end{equation*}
	Furthermore, we have 
	\begin{align*} X_+(t)V_1&=\begin{bmatrix}
	A_{\rm s}&B_{\rm s}
	\end{bmatrix}\begin{bmatrix}
	X_-(t)\\U_-(t)
	\end{bmatrix}V_1=A_{\rm s}+B_{\rm s}K=A_{\rm m},\\
	X_+(t)V_2&=\begin{bmatrix}
	A_{\rm s}&B_{\rm s}
	\end{bmatrix}\begin{bmatrix}
	X_-(t)\\U_-(t)
	\end{bmatrix}V_2=B_{\rm s}L=B_{\rm m},
	\end{align*}
	implying \eqref{V12}, that is, \eqref{image} holds.
	Then, according to Lemma~\ref{LDIMRC}, $ (U_-(t),X(t)) $ are informative for model reference control.
\end{proof}

Proposition \ref{TMC} shows that, when a solution of the matching equations exists, the full rank condition \eqref{frc} is a sufficient condition for informativity for model reference control.
The following theorem provides conditions under which the full rank condition \eqref{frc} becomes a necessary condition for informativity for model reference control.

\begin{theorem}\label{TBm}
	Assume there exists a controller gain pair $(K,L) \in \mathbb{R}^{m \times n} \times \mathbb{R}^{m \times p}$ that satisfies the matching equations \eqref{mc}.
	The following two statements are equivalent: 
	\vspace{-2ex}
	\begin{itemize}
		\item[(i)] The implication
		\begin{equation*}
		\im\!\begin{bmatrix}
		I&\hspace{-5pt} 0\\A_{\rm m}&\hspace{-5pt}B_{\rm m}
		\end{bmatrix}\!\!\subseteq\!
		\im\!\begin{bmatrix}
		X_-(t) \\ X_+(t)
		\end{bmatrix} \!\Rightarrow
		\rank\!\left(\!\begin{bmatrix}
		X_-(t)\\U_-(t)
		\end{bmatrix}\!\right)\!\!=\!n\!+\!m
		\end{equation*}
		holds for any $t \in \mathbb{N}$ and any $ (U_-(t),X(t))\in \mathfrak{B}_t $.
		\item[(ii)] $ p=m $ and $ B_{\rm m} $ has full column rank.
	\end{itemize}
\end{theorem}

\begin{proof}
	\underline{(i)$ \Rightarrow $(ii):}\\
	First, we claim that $ \rank(B_{\rm m})=m $.
	Suppose on the contrary that $\rank (B_{\rm m}) \neq m$. 
	Then, since $p \leq m$, we have that $\rank (B_m) < m$.
	Let $(K,L)$ be a solution of the matching equations \eqref{mc}. 
	We may assume without loss of generality that $\rank(L) < m$. 
	Indeed, by hypothesis, $\im B_{\rm m} \subseteq \im B_{\rm s}$ and thus $L := B_{\rm s}^\dagger B_{\rm m}$ satisfies the equation $B_{\rm s} L = B_{\rm m}$, where $B_{\rm s}^\dagger$ denotes the Moore-Penrose pseudoinverse of $B_{\rm s}$. 
	Note that the matrix $L$ satisfies $\rank (L) \leq \rank (B_{\rm m}) < m$.
	\\
	\indent
	Consider a set of data generated by the reference model \eqref{r}, which we denote by
	\begin{equation}\label{rd}
	\begin{aligned}
	R_-(t):=&\begin{bmatrix}
	r(0)&r(1)&\cdots&r(t-1)
	\end{bmatrix},\\
	X^{\rm m}(t):=&\begin{bmatrix}
	x_{\rm m}(0)&x_{\rm m}(1)&\cdots&x_{\rm m}(t)
	\end{bmatrix},\\
	X_-^{\rm m}(t):=&\begin{bmatrix}
	x_{\rm m}(0)&x_{\rm m}(1)&\cdots&x_{\rm m}(t-1)
	\end{bmatrix},\\
	X_+^{\rm m}(t):=&\begin{bmatrix}
	x_{\rm m}(1)&x_{\rm m}(2)&\cdots&x_{\rm m}(t)
	\end{bmatrix},\\
	\end{aligned}
	\end{equation}
	with $ t\in \mathbb{N} $ and $ R_-(t) $ persistently exciting of order $ n+1 $.	
	We have $$ X_+^{\rm m}(t)=A_{\rm m}X_-^{\rm m}(t)+B_{\rm m}R_-(t), $$ implying that
	\begin{equation}\label{m}
	\begin{bmatrix}
	X_-^{\rm m}(t) \\ X_+^{\rm m}(t)
	\end{bmatrix}
	=\begin{bmatrix}
	I&0\\A_{\rm m}&B_{\rm m}
	\end{bmatrix}
	\begin{bmatrix}
	X_-^{\rm m}(t)\\R_-(t)
	\end{bmatrix}.
	\end{equation}
	Because $ (A_{\rm m},B_{\rm m}) $ is controllable, it follows from Lemma \ref{LW} that
	\begin{equation*}
	\rank\left(\begin{bmatrix}
	X_-^{\rm m}(t)\\R_-(t)
	\end{bmatrix}\right)=n+p.
	\end{equation*}
	Therefore, \eqref{m} implies that
	\begin{equation}\label{xm}
	\im \begin{bmatrix}
	X_-^{\rm m}(t) \\ X_+^{\rm m}(t)
	\end{bmatrix}
	=
	\im \begin{bmatrix}
	I&0\\A_{\rm m}&B_{\rm m}
	\end{bmatrix}.
	\end{equation}
	Since $ (K,L) $ is a solution of the matching equations \eqref{mc}, the equation \eqref{m} implies that
	\begin{equation}
	\begin{bmatrix}
	X_-^{\rm m}(t) \\ X_+^{\rm m}(t)
	\end{bmatrix}
	=
	\begin{bmatrix}
	I&0\\A_{\rm s}&B_{\rm s}
	\end{bmatrix}
	\begin{bmatrix}
	I&0\\K&L
	\end{bmatrix}
	\begin{bmatrix}
	X_-^{\rm m}(t)\\R_-(t)
	\end{bmatrix}.
	\end{equation}
	By defining the input data
	$ U_-(t)=KX_-^{\rm m}(t)+LR_-(t) $,
	we obtain
	\begin{equation}
	X_+^{\rm m}(t)=A_{\rm s}X_-^{\rm m}(t)+B_{\rm s}U_-(t),
	\end{equation}
	implying that $(U_-(t),X^{\rm m}(t)) \in \mathfrak{B}_t$, i.e., the data $ (U_-(t),X^{\rm m}(t)) $ can be generated by the system $ (A_{\rm s},B_{\rm s}) $.
	However, since $ \rank (L) < m$ we have
	\begin{equation*}
	\rank\left(\begin{bmatrix}
	I&0\\K&L
	\end{bmatrix}\right)<n+m,
	\end{equation*}
	and therefore
	\begin{equation}\label{xmu}
	\rank\left(\begin{bmatrix}
	X_-^{\rm m}(t)\\U_-(t)
	\end{bmatrix}\right)<n+m.
	\end{equation}
	By \eqref{xm}, we have that \eqref{image}, but not \eqref{frc}, holds for the data $ (U_-(t),X^{\rm m}(t)) $, resulting in a contradiction.
	We conclude that $ \rank(B_{\rm m})=m $.
	Since $ \rank(B_{\rm m})\leq p \leq m $, we have that $ p=m $ and $ B_{\rm m} $ has full column rank.
		
	\underline{(ii)$ \Rightarrow $(i):}\\
	Suppose that $ p=m $ and $ B_{\rm m} $ has full column rank.
	Consider a trajectory $ (U_-(t),X(t))\in\mathfrak{B}_t $ satisfying \eqref{image}.
	Then, there exist matrices $ V_1\in\mathbb{R}^{t\times n} $ and $ V_2\in\mathbb{R}^{t\times p} $ such that
	\begin{equation*}
	\begin{aligned}
	\begin{bmatrix}
	I&0\\A_{\rm m}&B_{\rm m}
	\end{bmatrix}
	=&
	\begin{bmatrix}
	X_-(t) \\ X_+(t)
	\end{bmatrix}
	\begin{bmatrix}
	V_1&V_2
	\end{bmatrix}\\
	=&
	\begin{bmatrix}
	I&0\\A_{\rm s}&B_{\rm s}
	\end{bmatrix}
	\begin{bmatrix}
	X_-(t) \\ U_-(t)
	\end{bmatrix}
	\begin{bmatrix}
	V_1&V_2
	\end{bmatrix}.
	\end{aligned}
	\end{equation*}
	We have
	\begin{equation*}
	n+p= \rank \left(\begin{bmatrix}
	I&0\\A_{\rm m}&B_{\rm m}
	\end{bmatrix}\right) \leq \rank \left(\begin{bmatrix}
	X_-(t)\\U_-(t)
	\end{bmatrix}\right) \leq n+m,
	\end{equation*}
	implying that \eqref{frc} holds.
\end{proof}
	
	Theorem \ref{TBm} 
	establishes the special cases when unique system identification becomes necessary to solve the data-driven model reference control problem, which can be verified by simply checking the size and rank of $ B_{\rm m} $.
	We refer the reader to the discussion in Section \ref{excitation} on how Theorem \ref{TBm} offers insightful views on existing solutions to the MRAC problem.

	The weaker requirements on collected data of the inclusion \eqref{image} as compared to the full rank condition \eqref{frc} motivate us to develop a new MRAC approach that only exploits data informativity for model reference control, \emph{without} relying on data informativity for system identification.
	This solution will be the subject of the next section.

\section{MRAC with Online Informative Data}

To solve Problem~\ref{P}, we first present necessary and sufficient conditions for the existence of an adaptive law that guarantees asymptotic convergence of $\hat{K}(t)$ and $\hat{L}(t)$ in \eqref{controller} to some (possibly non-unique) solution of the matching equations \eqref{mc} (Section~V-A).
We then develop an input design only relying on data informativity for model reference control \mbox{(Section~V-B)}. 
Next, based on this input design, we establish convergence of MRAC to a solution of the matching equations (Section~V-C).
Finally, stability and boundedness for the closed-loop signals are studied (Section V-D) and discussions on existing MRAC solutions are given (Section V-E).

\subsection{Necessary and sufficient conditions for existence}

We aim to address the first part of Problem~\ref{P}, that is, establish necessary and sufficient conditions for the existence of an adaptive law and associated adaptive gains such that \eqref{converge} holds.
Existence is analyzed from two perspectives: a data-driven perspective based on the data generated by the system, and a model-based perspective based on the system matrices.
	

\begin{theorem}\label{Texist}
	The following statements are equivalent:
	\begin{itemize}
		\item[(i)] There exist an adaptive law as in \eqref{pTheta} and associated adaptive gains as in \eqref{hatKL} such that \eqref{converge} holds;
		\item[(ii)] There exist $ t\in\mathbb{N} $ and $ (U_-(t),X(t))\in\mathfrak{B}_t $ such that $ (U_-(t),X(t)) $ are informative for model reference control;
		\item[(iii)] There exist $K \in \mathbb{R}^{m \times n} $ and $ L \in \mathbb{R}^{m \times p}$ satisfying the matching equations \eqref{mc}.
	\end{itemize}
\end{theorem}


\begin{proof}
	\underline{(i)$\Rightarrow$(ii)}\\
	Suppose there exist an adaptive law as in \eqref{pTheta} and associated adaptive gains as in \eqref{hatKL} such that \eqref{converge} holds.
	Then, $$ \lim_{t\to\infty} B_{\rm s}\hat{K}(t) = A_{\rm m}-A_{\rm s} \text{ and } \lim_{t\to\infty} B_{\rm s}\hat{L}(t) = B_{\rm m}. $$
	Since for all $ t\in\mathbb{N} $, $$ \im B_{\rm s}\hat{K}(t)\subseteq\im B_{\rm s} \text{ and } \im B_{\rm s}\hat{L}(t)\subseteq\im  B_{\rm s} $$ 
	and $ \im B_{\rm s} $ is closed, we have $$ \im (A_{\rm m}-A_{\rm s})\subseteq \im B_{\rm s} \text{ and } \im B_{\rm m}\subseteq\im  B_{\rm s}. $$
	Therefore, there exists a pair $ (K,L) $ such that $ A_{\rm m}-A_{\rm s} = B_{\rm s}K $ and $ B_{\rm m} = B_{\rm s}L $, implying that \eqref{mc} holds.
	Next, we consider a set of data \eqref{rd} generated by the reference model \eqref{r} with $ R_-(t) $ persistently exciting of order $ n+ 1 $.
	Let \mbox{$ U_-(t)=KX_-^m(t)+LR_-(t) $}.
	Along similar lines as the proof (i)$ \Rightarrow $(ii) of Theorem \ref{TBm}, we have that $ (U_-(t),X^{\rm m}(t))\in\mathfrak{B}_t $ satisfy \eqref{xm}.
	Then, Lemma \ref{LDIMRC} implies statement (ii).
	
	\underline{(ii)$\Rightarrow$(iii)}\\
	Suppose there exist $ t\in\mathbb{N} $ and $ (U_-(t),X(t))\in\mathfrak{B}_t $ such that $ (U_-(t),X(t)) $ are informative for model reference control.
	Then, there exist $ V_1 $ and $ V_2 $ such that \eqref{V12} holds.
	Let \mbox{$ K=U_-V_1 $} and $ L=U_-V_2 $, then
	\begin{align*}
	A+BK&=AX_-V_1+BU_-V_1=X_+V_1=A_{\rm m},\\
	BL&=AX_-V_2+BU_-V_2=X_+V_2=B_{\rm m},
	\end{align*}
	implying statement (iii).
	
	\underline{(iii)$\Rightarrow$(i)}\\
	This implication follows from Theorem \ref{TMRAC} in Section \ref{Sconverge}.
\end{proof}


\subsection{Design of  the adaptive law and the input}
\label{Sdesign}
	
	We now provide a recipe for constructing $ i_t,j_t,\alpha_t,\beta_t $ and $ v $ in Problem~\ref{P}.
%
	To achieve online adaptive control in the data informativity framework, it is important to find the first time instant $ t $ at which the data $ (U_-(t),X(t))\in\mathfrak{B}_t $ are informative for model reference control.
	This motivates the following definition.
	\begin{definition} \label{DIT}
		Consider the input-state trajectory $ (u,x) \!\in\! \mathfrak{B} $.
		Then,
		\begin{equation}\label{T^*}
		\begin{aligned}
		T^*:=\min\{t\mid\, &(U_-(t),X(t))\in\mathfrak{B}_t \text{ are informative}\\
		&\text{ for model reference control} \}.
		\end{aligned}
		\end{equation}
	By convention, $ T^*=\infty $ if there does not exist a $ t \in \mathbb{N} $ for which $ (U_-(t),X(t)) $ are informative for model reference control.
	\end{definition}

	We will now show that if the trajectory $ (u,x) \in \mathfrak{B} $ is generated online via an appropriate estimation and control mechanism, then the associated $T^*\leq n+m$ if and only if the matching equations \eqref{mc} have a solution.
	The estimation and control mechanism will make use of the data matrices
	\begin{equation}
	\label{PhiU}
	\Phi_{U_-}\!(t)\!=\!\left\{
	\begin{aligned}
	& U_-(t) &&\text{if } 0<t\leq T^*+1 \\
	& 
	\setlength{\arraycolsep}{2pt}
	\!\! \begin{bmatrix}
	U_-(T^*)&u(t-1)
	\end{bmatrix}\! \!\, &&\text{if } t\!>\!T^*\!+\!1\, \text{and} \, |x(t)|\leq\sigma \\
	&\Phi_{U_-}(t-1) &&\text{otherwise},
	\end{aligned}
	\right.
	\end{equation}
	\begin{align}
	\nonumber
	\Phi_{X_-}\!(t)&\!=\!\left\{
	\begin{aligned}
	& X_-(t) &&\text{if } 0<t\leq T^*+1 \\
	& 
	\setlength{\arraycolsep}{2pt}
	\!\! \begin{bmatrix}
	X_-(T^*)&x(t-1)\\
	\end{bmatrix}\! \! &&\text{if } t\!>\!T^*\!+\!1\, \text{and} \, |x(t)|\leq\sigma \\
	& \Phi_{X_-}(t-1) &&\text{otherwise},
	\end{aligned}
	\right.\\ \nonumber
	\Phi_{X_+}\!(t)&\!=\!\left\{
	\begin{aligned}
	& X_+(t) &&\text{if } 0<t\leq T^*+1 \\
	& 
	\setlength{\arraycolsep}{2pt}
	\!\! \begin{bmatrix}
	X_+(T^*)&x(t)\\
	\end{bmatrix}\! \qquad \!\! \!\!\  &&\text{if } t\!>\!T^*\!+\!1\, \text{and} \, |x(t)|\leq\sigma \\
	& \Phi_{X_+}(t-1) &&\text{otherwise}, 
	\end{aligned}
	\right.\\
	\label{PhiX}
	\Phi_X(t)&\!=\! \begin{bmatrix}
	\Phi_{X_-}(t) \\ \Phi_{X_+}(t)
	\end{bmatrix}\!,
	\end{align}
	where $ \sigma\! \in\! \mathbb{R} $ is a given upper bound on the norm of the collected state data.
	The use of data matrices with past and current data is inspired by the concurrent learning approach \cite{chowdhary2010concurrent}, where past and current data are concurrently used to update the control gains.
	Note that to construct the data matrices in \eqref{PhiU} and \eqref{PhiX}, we need to verify at every time $t$ whether $t \leq T^* + 1$. 
	We emphasize that, even though the exact value of $T^*$ may be unknown at time $t$, one can verify whether $t \leq T^* + 1$ via the linear algebraic condition (11).
	
	Moreover, we consider a matrix $ \Theta(t)\in \mathbb{R}^{i_t\times(n+p)} $, where 
	\begin{equation}\label{it}
	i_t=\left\{
	\begin{aligned}
	& 1 &&\text{if } t=0 \\
	& t &&\text{if } 0<t\leq T^* \\
	& T^*+1 &&\text{if } t>T^*.
	\end{aligned}
	\right.
	\end{equation}
	We let $ \Theta(0)=\Theta(1) \in \mathbb{R}^{1\times(n+p)} $ be arbitrary and consider the following adaptive update law for $ \Theta $:
	\begin{equation}\label{Theta}
	\Theta(t+1)=\left\{
	\begin{aligned}
	& \begin{bmatrix}
	\Theta(t)-\gamma\Phi_X(t)^\top\Delta(t) \\0
	\end{bmatrix}
	&&\text{if } 0<t\leq T^*\\
	& \ \Theta(t)
	-\gamma\Phi_X(t)^\top\Delta(t) &&\text{if } t>T^*,
	\end{aligned}
	\right.
	\end{equation}
	where  $ 0<\gamma<2 $  and
	\begin{equation}\label{Delta}
	\begin{aligned}
	&\Delta(t)=\\
	&\left\{
	\begin{aligned}
	& \Phi_X(t)\Theta(t)-
	\setlength{\arraycolsep}{2pt}
	\begin{bmatrix}
	I&0\\A_{\rm m}&B_{\rm m}
	\end{bmatrix}
	&&\text{if } 0<t\leq T^* \\
	& \frac{1}{\|\Phi_X(t)\|_F^2}
	\left(\Phi_X(t)\Theta(t)-
	\setlength{\arraycolsep}{3pt}
	\begin{bmatrix}
	I&0\\A_{\rm m}&B_{\rm m}
	\end{bmatrix}\right) &&\text{if } t>T^*.
	\end{aligned}
	\right.
	\end{aligned}
	\end{equation}
Note that $ \|\Phi_X(t)\|_F^2 \neq 0 $  for all $ t> T^* $.
	
Given the adaptive law in \eqref{Theta}, we introduce the control input formally in the next lemma, where we also clarify the relation between $T^*$ and the existence of solutions of the matching equations \eqref{mc}.
	
\begin{lemma}\label{LT}
	Let $ u(0)\in \mathbb{R}^m $  be nonzero and $ x(0)\in \mathbb{R}^n $.
	For $t \in \mathbb{N}$, design the input $ u(t) $ as follows:
	\begin{itemize}
	\item If 
	\begin{itemize}
		\item[i)] $ t<n+m $,
			\begin{equation}\label{ua}
			\!\!\!\!\!\!\!\!\!\!\;
		\text{ii)} \
			\begin{bmatrix}
			[I\qquad0]\\\Phi_{U_-}(t)\Theta(t)
			\end{bmatrix}
			\begin{bmatrix}
			x(t)\\r(t)
			\end{bmatrix}\in
			\im	\begin{bmatrix}
			\Phi_{X_-}(t)\\\Phi_{U_-}(t)
			\end{bmatrix}\!, \text{ and }
			\end{equation}
		\item[iii)] the data $ (U_-(t),X(t)) $ are not informative for model reference control,
	\end{itemize}
	then there exists a $ w \in \mathbb{R}^m $ such that
	\begin{equation}\label{xi}
		\rank\left(\begin{bmatrix}
		\Phi_{X_-}(t)&x(t) \\
		\Phi_{U_-}(t)&w
		\end{bmatrix}\right)=
		\rank\left(\begin{bmatrix}
		\Phi_{X_-}(t)\\\Phi_{U_-}(t)
		\end{bmatrix}\right)+1.
	\end{equation}
	In this case, select $ u(t) $ as in \eqref{controller} with 
	\begin{equation}\label{xib}
		\hat{K}(t)=0, \hat{L}(t)=0 \text{ and } v(t)=w.
	\end{equation}
	\item Otherwise, select $ u(t) $ as in \eqref{controller} with 
	\begin{equation}\label{gain}
		\begin{bmatrix}
		\hat{K}(t)&\hat{L}(t)
		\end{bmatrix}=\Phi_{U_-}(t)\Theta(t) \text{ and } v(t)=0.
	\end{equation}
	\end{itemize}
	Consider the resulting input-state trajectory $ (u,x) \in \mathfrak{B} $.
	Then, $ T^*\leq n+m $ if and only if \eqref{mc} holds for some $ (K, L) $.
\end{lemma}
	
\begin{proof}
	We first show that if i), ii) and iii) hold, then there exists a $ w \in \mathbb{R}^m$ satisfying \eqref{xi}. 
	To do so, note that iii) implies that $t \leq T^* + 1$ so $\Phi_{U_-}(t) = U_-(t)$ and $\Phi_{X_-}(t) = X_-(t)$. 
	Moreover, ii) implies that $x(t) \in \im X_-(t)$. 
	Therefore, it follows from \cite[Theorem 1]{van2021beyond} 
	that there exist a $ w $ such that \eqref{xi} holds.
	
	Now we show that $ T^*\leq n+m $ if and only if condition (iii) in Theorem~\ref{Texist} holds.
	
	For necessity, we suppose that $ T^*\leq n+m $.
	According to Definition \ref{DIT} and Lemma \ref{LDIMRC},  there exist gains $ (K,L) $ such that $ \Sigma_{(U_-(T^*),X(T^*))}\subseteq \Sigma^{K,L} $.
	Since $ (A_{\rm s},B_{\rm s}) \in \Sigma_{(U_-(T^*),X(T^*))} $, we have that \eqref{mc} holds.
	
	For sufficiency, we suppose on the contrary that $T^* > n+m$. 
	We will prove that
	\begin{equation}\label{nm}
	\rank \begin{bmatrix} X_-(n+m) \\ U_-(n+m) \end{bmatrix} = n+m.
	\end{equation}
	To do so, we note that 
	\begin{equation*}
	\rank \begin{bmatrix} X_-(1) \\ U_-(1) \end{bmatrix} = 1
	\end{equation*}
	because $u(0)$ is nonzero. 
	Next, we show that 
	\begin{equation}\label{rank+}
		\text{rank}\left(\begin{bmatrix}
	X_-(t'+1)\\U_-(t'+1)
	\end{bmatrix}\right)=
	\text{rank}\left(\begin{bmatrix}
	X_-(t')\\U_-(t')
	\end{bmatrix}\right)+1
	\end{equation}
	for all $t' = 1,2,...,n+m-1$. 
	Let $t' \in \{1,2,...,n+m-1\}$. 
	Since $t' \leq T^*+1$, the data $(U_-(t'),X(t'))$ are not informative for model reference control and
	\begin{equation*}
	\Phi_{U_-(t')} = U_-(t') \text{ and } \Phi_{X_-(t')} = X_-(t').
	\end{equation*}
	Now, if \eqref{ua} holds for $ t=t' $,
	then \eqref{rank+} follows from selecting $u(t')$ as in \eqref{xib} for $ t=t' $.
	Otherwise, if \eqref{ua} does not hold for $ t=t' $, then \eqref{rank+} follows from selecting $u(t')$ as in \eqref{gain} for $ t=t' $.
	This proves that \eqref{nm} holds. 
	However, by Lemma \ref{LDISI} and Proposition \ref{TMC}, we see that $(U_-(n+m),X_-(n+m))$ are informative for model reference control. 
	This contradicts the definition of $T^*$. Therefore, $T^* \leq n+m$.
	\end{proof}

Lemma \ref{LT} demonstrates that, with the specific choice of the input, either $ T^*\leq n+m $  or $ T^*=\infty $.
In the latter case, there are no gains $ (K, L) $ such that \eqref{mc} holds.
Therefore, Lemma \ref{LT} provides an effective method to check whether the matching equations \eqref{mc} have a solution.
Indeed, $ T^*=\infty $ if \eqref{image} does not hold for $ t=n+m $.
In other words, $ T^* $ can always be determined after collecting at most $ n+m $ data samples.

\begin{remark}
	(Determining $T^*$).
	To determine $T^*$, one needs to verify the linear algebraic condition \eqref{image}. 
	Note that \eqref{image} is equivalent to
	\begin{equation}\label{rank=}
	\rank\left(\begin{bmatrix} X_-(t) \\ X_+(t) \end{bmatrix}\right)= \rank\left(\begin{bmatrix} X_-(t)&I&0\\X_+(t)&A_{\rm m}&B_{\rm m} \end{bmatrix}\right),
	\end{equation}
	which provides a computationally efficient way to verify \eqref{image}.
	Furthermore, \eqref{rank=} implies that $ n+\rank(B_{\rm m})\leq T^*\leq n+m $, so it suffices to check \eqref{rank=} only for $t \geq n+\rank(B_{\rm m})$.
\end{remark}

	\begin{remark} (Design of the input).
	Following the approach in \cite[Theorem 1]{van2021beyond}, one way to design $w$ is as follows. 
	First, find \mbox{$ \xi\in \mathbb{R}^n $}, $\eta\in \mathbb{R}^m\setminus \{0\} $ such that 
	\begin{equation}\label{v}
	\begin{bmatrix} \xi^\top  & \eta^\top \end{bmatrix} \begin{bmatrix} \Phi_{X_-}(t) \\ \Phi_{U_-}(t) \end{bmatrix} = 0.
	\end{equation}
	Then, choose $w$ such that 
	\begin{equation}\label{xieta}
	\xi^\top x(t) + \eta^\top w \neq 0.
	\end{equation}
\end{remark}

\vspace{0.6em}
\begin{remark} (Design of the parameters).
	In \eqref{PhiU} and \eqref{PhiX}, $ \sigma $ can be designed, for instance, as an upper bound of $ x_{\rm m} $, which can be determined since the reference model $ A_{\rm m}, B_{\rm m} $ and the reference input $r \in \ell_\infty^p$ are given.
	In \eqref{Theta}, the gain has been taken as a scalar $ \gamma \in (0,2) $, for simplicity\footnote{
		Requirements analogous to $ \gamma \in (0, 2) $ are common in discrete-time gradient-based adaptation \cite[Lemma 4.1.1]{ioannou2006adaptive}. 
		It will be shown in the proof of Theorem \ref{TMRAC} that this choice of $ \gamma $ also ensures the existence of a limit of the Lyapunov function, which in turn guarantees \eqref{converge}.
	}.
	However, the scalar $\gamma$ can be generalized to a matrix gain \mbox{$\Gamma= \Gamma^\top$} satisfying $ 0 < \Gamma < 2I $, which allows different adaptation rates along different directions \cite{ioannou2006adaptive}.
\end{remark}

\subsection {Convergence to a matching solution}
\label{Sconverge}

We are now in a position to  formulate the convergence properties for the proposed recipe.

\begin{theorem}\label{TMRAC}
	If the matching equations \eqref{mc} have a solution, the adaptive gains designed as in \eqref{xib}-\eqref{gain} are such that \eqref{converge} holds.
\end{theorem}
	
\begin{proof}
	By Lemma \ref{LT}, we have $ T^*\leq n+m $.
	Then, by definition of $ \Phi_X(t) $, there exists a matrix $ \Theta^* \in \mathbb{R}^{(T^*+1)\times(n+p)} $ such that
	\begin{equation}\label{Theta*}
		 \Phi_X(t)\Theta^*=\begin{bmatrix}
		I&0\\A_{\rm m}&B_{\rm m}
		\end{bmatrix}
	\end{equation} 
	for all $ t > T^* $.
	Let $ t>T^* $ and denote $ \tilde{\Theta}(t)=\Theta(t)-\Theta^* $.
	According to \eqref{Theta}, we obtain the dynamics
	\begin{equation}\label{Thetat}
		\tilde{\Theta}(t+1)=\tilde{\Theta}(t)-\gamma\Phi_X(t)^\top\Delta(t).
	\end{equation}
	We now proceed to define the candidate Lyapunov function $ V:\mathbb{R}^{(T^*+1)\times(n+p)}\to\mathbb{R}_+ $ as \begin{equation}\label{lf}
			V(\tilde{\Theta})=\frac{1}{2}\tr\left(\tilde{\Theta}^\top\tilde{\Theta}\right).
	\end{equation}
	Note that $ V(\tilde{\Theta})\geq0 $ and $ V(\tilde{\Theta})=0 \Leftrightarrow \tilde{\Theta}=0 $.
	The rate of change of the Lyapunov function along the trajectories of \eqref{Thetat} is given by:
	\begin{align*}
		&V(\tilde{\Theta}(t+1))-V(\tilde{\Theta}(t))\\
		=&\frac{1}{2}\tr\left(\tilde{\Theta}(t+1)^\top\tilde{\Theta}(t+1)\right)
		-
		\frac{1}{2}\tr\left(\tilde{\Theta}(t)^\top\tilde{\Theta}(t)\right)\\
		=&-\gamma\tr\left(\tilde{\Theta}(t)^\top\Phi_X(t)^\top\Delta(t)\right)\\
		&+	\frac{\gamma^2}{2}\tr\left(\Delta(t)^\top\Phi_X(t)\Phi_X(t)^\top\Delta(t)\right)\\
		=&-\gamma\|\Phi_X(t)\|_F^2 \cdot \tr\left(\Delta(t)^\top\Delta(t)\right) \\
		&+\frac{\gamma^2}{2}\tr\left(\Delta(t)^\top\Phi_X(t)\Phi_X(t)^\top\Delta(t)\right)\\
		=&-\|\Phi_X(t)\|_F^2
		\cdot\tr\left(\Delta(t)^\top
		\left(\gamma I-\frac{\gamma^2\Phi_X(t)			\Phi_X(t)^\top}{2\|\Phi_X(t)\|_F^2}
		\right)
		\Delta(t)\right)\!.
	\end{align*}
	Since $ 0 < \gamma < 2 $, there exists $ \varepsilon>0 $ such that $\gamma - \frac{\gamma^2}{2} \geq \varepsilon$.
	We claim that
	\begin{equation}\label{pd}
		\gamma I-\frac{\gamma^2\Phi_X(t)\Phi_X(t)^\top}{2\|\Phi_X(t)\|_F^2}\geq\varepsilon I.
	\end{equation}
	Since $ \|\Phi_X(t)\|_F\geq\|\Phi_X(t)\|_2 $, we have that for any $ v\in\mathbb{R}^{2n} $,
\begin{align*}
	v^\top\!\left(\gamma I-\frac{\gamma^2\Phi_X(t)	\Phi_X(t)^\top}{2\|\Phi_X(t)\|_F^2}\right)v
	& \geq \gamma|v|^2-\frac{\gamma^2\|\Phi_X(t)\|_2^2}{2\|\Phi_X(t)\|_F^2}|v|^2\\
	& \geq \gamma|v|^2-\frac{\gamma^2}{2}|v|^2 \geq \varepsilon |v|^2.
\end{align*}
	As a result, \eqref{pd} indeed holds.
	Therefore,
	\begin{equation}	\label{V}
		V(\tilde{\Theta}(t+1))-V(\tilde{\Theta}(t))	\leq -\varepsilon\|\Phi_X(t)\|_F^2\cdot	\|\Delta(t)\|_F^2 \leq 0.
	\end{equation}
	By \eqref{lf}, we have that $ \{V(\tilde{\Theta}(t))\}^{\infty}_{t=T^*+1} $ is a non-increasing sequence that is bounded from below.
	Thus, the limit \mbox{$ V_\infty:=\lim\limits_{t\to\infty}V(\tilde{\Theta}(t)) $} exists.
	Moreover, from \eqref{V} we have
	\begin{equation}\label{Vlimit}
		\varepsilon\sum_{t=T^*+1}^{\infty}
		\|\Phi_X(t)\|_F^2\cdot
		\|\Delta(t)\|_F^2\leq V(T^*+1)-V_\infty,
	\end{equation}
	implying that
	\begin{equation}\label{lim}
		\lim\limits_{t\to\infty}
		\|\Phi_X(t)\|_F^2\cdot
		\|\Delta(t)\|_F^2
		=\lim\limits_{t\to\infty}
		\frac{\|\Phi_X(t)\tilde{\Theta}(t)\|_F^2}
		{\|\Phi_X(t)\|_F^2}
		=0.
	\end{equation}
	From \eqref{PhiX}, there exists a $c \!>\! 0$ such that $ \tr\!\left(\Phi_X(t)^\top\Phi_X(t)\right)\!\leq\! c $ for all $ t>T^* $.
	Therefore, 
	\begin{align*}
		\frac{\|\Phi_X(t)\tilde{\Theta}(t)\|_F^2}{\|\Phi_X(t)\|_F^2}
		\geq \frac{1}{c} \|\Phi_X(t)\tilde{\Theta}(t)\|_F^2
	\end{align*}
	for all $ t>T^* $, which, by \eqref{lim}, implies that \begin{equation*}
		\lim\limits_{t\to\infty}\|\Phi_X(t)\tilde{\Theta}(t)\|_F^2=0.
	\end{equation*}
	Hence,
	\begin{equation*}
		\lim\limits_{t\to\infty}\Phi_X(t)\tilde{\Theta}(t)=0,
	\end{equation*}
	or, equivalently,
	\begin{equation}\label{convergence}
		\lim\limits_{t\to\infty}\Phi_X(t)\Theta(t)=
		\begin{bmatrix}
		I&0\\A_{\rm m}&B_{\rm m}
		\end{bmatrix}.
	\end{equation}
	Using the fact that $ \Phi_{X_+}(t) = A_{\rm s} \Phi_{X_-}(t) + B_{\rm s}\Phi_{U_-}(t) $, and using the definition of $ \hat{K}(t) $ and $ \hat{L}(t) $ in \eqref{gain}, we have
\begin{equation}\label{convergeAmBm}
	\begin{aligned}
		&\lim\limits_{t\to\infty}
		\begin{bmatrix}
		A_{\rm s}+B_{\rm s}\hat{K}(t) &B_{\rm s}\hat{L}(t)
		\end{bmatrix}\\
		=&\begin{bmatrix}
		A_{\rm s}&0
		\end{bmatrix}+\lim\limits_{t\to\infty}B_{\rm s}\Phi_{U_-}(t)\Theta(t)\\
		=&\begin{bmatrix}
		A_{\rm s}&0
		\end{bmatrix}+
		\lim\limits_{t\to\infty}
		(\Phi_{X_+}(t) - A_{\rm s}\Phi_{X_-}(t)) \Theta(t)\\
		=&\begin{bmatrix}
		A_{\rm s}&0
		\end{bmatrix}+
		\begin{bmatrix}
		A_{\rm m}-A_{\rm s}&B_{\rm m}
		\end{bmatrix}
		=\begin{bmatrix}
		A_{\rm m}&B_{\rm m}
		\end{bmatrix}.
	\end{aligned}
\end{equation}
We conclude that \eqref{converge} holds.
\end{proof}

Theorem~\ref{TMRAC} shows that, provided the matching equations \eqref{mc} admit a solution, the input design in Lemma~\ref{LT} guarantees convergence to a solution of the matching equation by only exploiting data that are informative for model reference control.
Notably, the proposed MRAC approach requires neither unique system identification nor structural assumptions on the system apart from controllability.
We refer the reader to Section \ref{excitation} for a discussion on structural assumptions in existing MRAC approaches, such as square assumptions on $ L $ ($p=m$), or $B_{{\rm s}}$ being known, or $B_{{\rm s}}$ being full column rank.
Before this, we discuss the boundedness of the closed-loop signals and the convergence of the error.


\subsection{Stability and boundedness of the closed-loop signals}
\label{stability}

In this subsection, we show the boundedness of all closed-loop signals and the convergence of the tracking error.

To establish a proper foundation, we begin by revisiting the following notation.
A function $ f: \mathbb{R}_{+} \to \mathbb{R}_{+} $ is said to be a \textit{$ \mathcal{K} $ function }if it is increasing and $ f(0)=0 $.
A function $ g: \mathbb{R}_{+} \times \mathbb{R}_{+} \to \mathbb{R}_{+} $ is said to be a \textit{$ \mathcal{KL} $ function} if for each fixed $ t\geq0 $, the function $ g(\cdot,t) $ is a $ \mathcal{K} $ function, and for each fixed $ s\geq0 $, the function $ g(s,\cdot) $ is non-increasing and $ g(s,t)\to0 $ as $ t\to\infty $.
We then consider the following definitions of stability.
\begin{definition}\label{Des}
	($ \!\!\!\! $\cite{zhou2017asymptotic}).
	Consider a linear time-varying autonomous system
	\begin{equation}\label{as}
	x(t+1)=A(t)x(t),
	\end{equation}
	where $ x(t) \in \mathbb{R}^n $ is the state.
	We say that \eqref{as} is \textit{uniformly exponentially stable} if, for any $ t_0\in\mathbb{Z}_+ $, there exist scalars $ \rho>0 $ and $ \lambda\in(0,1) $ such that 
	\begin{equation*}
	|x(t)|\leq\rho|x(t_0)|\lambda^{t-t_0}\quad \forall t\geq t_0.
	\end{equation*}
\end{definition}
\vspace{0.6em}

	The following result is from \cite[Corollary 3]{zhou2017asymptotic}.
\begin{lemma}\label{LUES}
	Assume that $\lim_{t\to\infty}A(t)$ exists and $\lim_{t\to\infty}A(t)=A_\textrm{m}$, where $A_m$ is Schur.
	Then, \eqref{as} is uniformly exponentially stable.
\end{lemma}
\vspace{0.6ex}
\begin{definition}\label{Diss}
	($ \!\! $\cite{jiang2001input}).
	Consider a nonlinear system
	\begin{equation}\label{ns}
	x(t+1)=h(x(t),u(t)),
	\end{equation}
	where $ x(t) \in \mathbb{R}^n $ is the state, $ u(t) \in \mathbb{R}^m $ is the input, and the map $ h:  \mathbb{R}^{n}\times\mathbb{R}^{m}\to\mathbb{R}^{n} $ satisfies $ h(0,0)=0 $.
	We say that \eqref{ns} is \textit{input-to-state stable} if there exist a $ \mathcal{KL} $-function $ g:\mathbb{R}_{+} \times \mathbb{R}_{+} \to \mathbb{R}_{+} $ and a $ \mathcal{K} $-function $ f: \mathbb{R}_{+} \to \mathbb{R}_{+} $ such that, for any input $ u\in\ell_{\infty}^m $ and initial state $ x(0)\in \mathbb{R}^n $, it holds that
	\begin{equation}\label{iss}
	|x(t)|\leq g(|x(0)|,t)+f(\bar{u}),
	\end{equation}
	where $ \bar{u}\in \mathbb{R}_+ $ is such that $ |u(t)|\leq\bar{u} $, for all $ t\in \mathbb{Z}_+ $.
\end{definition}
\vspace{0.6ex}

Definition \ref{Diss} implies that if a system is input-to-state stable, then for any initial state $ x(0) $ and $ u\in \ell_{\infty}^m $, we have $ x\in\ell_{\infty}^n $.
We are now in a position to give the following stability and boundedness result.

\begin{theorem}\label{Pe}
	Assume that the matching equations \eqref{mc} have a solution.
	Let $x(0) \in \mathbb{R}^n$ and consider the input $u(t) \in \mathbb{R}^m$ for $t \in \mathbb{N}$ as in Lemma \ref{LT}.
	Then, the resulting input-state trajectory $ (u,x) \!\in\! \mathfrak{B} $ satisfies $ x\!\in\!\ell_{\infty}^n $ and $ u\!\in\! \ell_{\infty}^m $. 
	Moreover, the tracking error $ e=x-x_{\rm m} $ satisfies $e \!\in\! \ell_{\infty}^n $ and $ \lim_{t \rightarrow \infty} e(t) = 0 $.
\end{theorem}

\begin{proof}
	Let us first consider the auxiliary autonomous system
	\begin{equation}\label{ea}
	e'(t+1)=(A_{\rm s}+B_{\rm s}\hat{K}(t))e'(t).
	\end{equation}
	Since $ \lim_{t\to\infty} A_{\rm s}+B_{\rm s}\hat{K}(t) = A_{\rm m} $ and $ A_{\rm m} $ is Schur, Lemma~\ref{LUES} implies that \eqref{ea} is uniformly exponentially stable.
	From Definition \ref{Des}, we have that for any $ t_0\in\mathbb{Z}_+ $ and $ t>t_0 $, there exist scalars $ \rho'>0 $ and $ \lambda'\in(0,1) $ such that 
	\begin{equation}\label{e'}
	|e'(t)|=\big|(\prod_{i=t_0}^{t-1}(A_{\rm s}+B_{\rm s}\hat{K}(i)))e'(t_0)\big|\leq\rho' |e'(t_0)|(\lambda')^{t-t_0}.
	\end{equation}
	Next, let us express the dynamics of $ e $ as
	\begin{equation*}
	\begin{aligned}
	&e(t+1)=x(t+1)-x_{\rm m}(t+1)\\
	=&(A_{\rm s}+B_{\rm s}\hat{K}(t))x(t)+B_{\rm s}\hat{L}(t)r(t)-A_{\rm m}x_{\rm m}(t)-B_{\rm m}r(t)\\
	=&(A_{\rm s}+B_{\rm s}\hat{K}(t))e(t)\\
	+&(A_{\rm s}+B_{\rm s}\hat{K}(t)-A_{\rm m})x_{\rm m}(t)+(B_{\rm s}\hat{L}(t)-B_{\rm m})r(t).
	\end{aligned}
	\end{equation*}
	In other words, 
	\begin{equation}\label{e}
	e(t+1)=(A_{\rm s}+B_{\rm s}\hat{K}(t))e(t)+\hat{u}(t),
	\end{equation}
	where $ \hat{u}(t) $ is defined as
	\begin{equation}\label{hatu}
	 \hat{u}(t)=(A_{\rm s}+B_{\rm s}\hat{K}(t)-A_{\rm m})x_{\rm m}(t)+(B_{\rm s}\hat{L}(t)-B_{\rm m})r(t).
	\end{equation}
	Let us now show that $ \hat{u} \in \ell_{\infty}^n $. 
	Note that $ x_{\rm m}\in\ell_{\infty}^n $ since $ r\in\ell_{\infty}^p $ and $ A_{\rm m} $ is Schur.
	Also, according to \eqref{lf} and \eqref{V}, we have that there exists a constant $ c>0 $ such that $ \|\Theta(t)\|_2<c $ for all $ t>T^* $, which further implies $ \hat{K}\in\ell_{\infty}^{m,n} $ and $ \hat{L}\in\ell_{\infty}^{m,p} $ by \eqref{gain}.
	This implies that $ \hat{u}\in\ell_{\infty}^m $.
	Let $ \bar{u}\in \mathbb{R}_+ $ be such that $ |\hat{u}(t)| \leq \bar{u} $ for all $ t\in \mathbb{Z}_+ $.\\
	We now show that \eqref{e} is input-to-state stable.
	Using \eqref{e'}, there exist $ \rho_0,\rho_1,\dots,\rho_{t-1}>0 $ and $ \lambda_0,\lambda_1,\dots,\lambda_{t-1}\in(0,1) $ such that
	\begin{align*}
	|e(t)|=&\big|(\prod_{i=0}^{t-1}(A_{\rm s}+B_{\rm s}\hat{K}(i)))e(0)+\hat{u}(t-1)+\\
	&\sum_{j=1}^{t-1}(\prod_{i=j}^{t-1}(A_{\rm s}+B_{\rm s}\hat{K}(i)))\hat{u}(j-1)\big|\\
	\leq&\big|(\prod_{i=0}^{t-1}(A_{\rm s}+B_{\rm s}\hat{K}(i)))e(0)\big|+|\hat{u}(t-1)|+\\
	&\sum_{j=1}^{t-1}\big|(\prod_{i=j}^{t-1}(A_{\rm s}+B_{\rm s}\hat{K}(i)))\hat{u}(j-1)\big|\\
	\leq&\rho_0 |e(0)|\lambda_0^{t}+|\hat{u}(t-1)|+\sum_{j=1}^{t-1}\rho_j |\hat{u}(j-1)|\lambda_j^{t-j}\\
	\leq&\rho_0 |e(0)|\lambda_0^{t}+\bar{\rho}\bar{u}\sum_{j=0}^{t-1}\bar{\lambda}^j
	\leq \rho_0 |e(0)|\lambda_0^{t}+\frac{\bar{\rho}\bar{u}}{1-\bar{\lambda}},
	\end{align*}
	where $ \bar{\rho}\!=\!\max\{1,\rho_1,\dots,\rho_{t-1}\} $ and $ \bar{\lambda}\!=\!\max\{\lambda_1,\dots,\lambda_{t-1}\} $.$ \! $
	Given the $\mathcal{KL}$-function
	\begin{equation*}
	g(|e(0)|,t)=\rho_0 |e(0)| \lambda_0^{t}
	\end{equation*}
	and the $\mathcal{K}$-function
	\begin{equation*}
	f(\bar{u})=\frac{\bar{\rho}\bar{u}}{1-\bar{\lambda}},
	\end{equation*}
	we have
	\begin{equation*}
	e(t)\leq g(|e(0)|,t)+f(\bar{u}).
	\end{equation*}
	Hence, \eqref{e} is input-to-state stable with respect to the input $ \hat{u} $, implying, from Definition \ref{Diss}, that $ e\in\ell_{\infty}^n $.
	From the fact that $ x_{\rm m}\in\ell_{\infty}^n $, we further have $ x=e+x_{\rm m}\in\ell_{\infty}^n $.
	Then, by \eqref{controller}, we have $ u\in\ell_{\infty}^m $.
	\\
	Finally, we demonstrate the convergence of the error.
	According to \cite[Sect. 3.2]{jiang2001input}, for the input-to-state stable system \eqref{e}, there exists a $ \mathcal{K} $-function $ f' $ such that for any $ e(0)\in\mathbb{R}^{n} $,
	\begin{equation*}
	\limsup_{t\to\infty} |e(t)|\leq f'(\limsup_{t\to\infty} |\hat{u}(t)|).
	\end{equation*}
	From \eqref{hatu} and the fact that $ x_{\rm m}\in\ell_{\infty}^n $, $ r\in\ell_{\infty}^p $ and \eqref{converge} holds, we have $ \lim_{t \rightarrow \infty} \hat{u}(t) = 0 $, implying $ \lim_{t \rightarrow \infty} e(t) = 0 $.
\end{proof}

\subsection{Discussion on excitation and structural assumptions in state-of-the-art MRAC}
\label{excitation}

We now discuss the requirements on collected data imposed by the proposed MRAC method as compared to existing state-of-the-art MRAC methods.
Without going into full details, we mention that state-of-the-art approaches like concurrent learning or composite MRAC \cite{lavretsky2009combined,chowdhary2010concurrent,chowdhary2013concurrent,cho2018composite,lee2019concurrent,roy2018combined} are based on a Lyapunov design in continuous time, based on  the error dynamics
\begin{equation}\label{ec}
\dot{e}=A_{\rm m}e+B_{\rm s}\tilde{\Lambda}\psi,
\end{equation}
with $ A_{\rm m} $ the Hurwitz state matrix of the reference model, $ B_{\rm s} $ the control matrix of the system, $ \psi $ measured signals, and $ \tilde{\Lambda} $ a parameter estimation error. 
In this paper, we have proposed a \textit{discrete-time} MRAC method. 
Although a direct comparison to the state-of-the-art is thus not possible, some considerations can still be made based on the analogy between the error dynamics in \eqref{e} and \eqref{ec}.
To this end, note that all state-of-the-art solutions with the exception of \cite{roy2018combined} require the input matrix $ B_{\rm s} $ to be known, which is not required in the MRAC method in the current paper.
%
The estimation of $B_{\rm s}$ in \cite{roy2018combined} leveraged the concept of \textit{initial excitation} to relax the classical persistence of excitation conditions.
When written in discrete time, initial excitation requires the existence of $ T\in\mathbb{N} $ and $ \delta>0 $ such that
\begin{equation}\label{ie}
\sum_{t=0}^{T-1}
\begin{bmatrix}
x(t)\\u(t)
\end{bmatrix}
\begin{bmatrix}
x(t)\\u(t)
\end{bmatrix}^\top -\delta I > 0.
\end{equation}
Clearly, \eqref{ie} implies that the data
\begin{equation*}
	\begin{bmatrix}
	X_-(T)\\U_-(T)
	\end{bmatrix}=
	\begin{bmatrix}
	x(0)&x(1)&\cdots&x(T-1)\\
	u(0)&u(1)&\cdots&u(T-1)
	\end{bmatrix}
\end{equation*}
satisfy the full rank condition \eqref{frc}. Thus, from Lemma \ref{LDISI}, we conclude that the requirements imposed on collected data by initial excitation are such that $ (A_{\rm s},B_{\rm s}) $ can be uniquely identified. 
Such an observation further illustrates the significance of Proposition \ref{TMC} and Theorem \ref{TBm}.
Similar to \cite{roy2018combined}, let us assume that the matching equations \eqref{mc} have a solution and $ p=m $.
From Proposition \ref{TMC}, we know that \eqref{ie} implies data informativity for model reference control.
However, when $ B_{\rm m} $ does not have full column rank, Theorem \ref{TBm} implies that the initial excitation \eqref{ie} is sufficient but not necessary to solve the model reference control problem. 
This means that initial excitation imposes stronger conditions on collected data, not required in the proposed MRAC approach.
Indeed, adaptive approaches based on initial excitation or other relaxed excitation conditions \cite{lavretsky2009combined,chowdhary2010concurrent,chowdhary2013concurrent,cho2018composite,lee2019concurrent,roy2018combined} only prove the sufficiency of such conditions for converging to a solution of the matching equations.
In comparison to these works, the contribution of this paper is to establish necessary and sufficient conditions for converging to a solution of the matching equations.

The works \cite{kogan1994locally,kogan1996locally} study least squares estimation of the system parameters in the absence of persistence of excitation. 
These works show that, in this case, the parameters do not necessarily converge to the true system parameters. However, the closed-loop system matrix may converge to a reference matrix. 
The latter convergence, however, is not always guaranteed and depends on the quality of the data, which was not discussed in \cite{kogan1994locally,kogan1996locally}. 
In comparison to these works, the contribution of this paper is to highlight the required conditions on the closed-loop data under which the adaptive gains converge to a solution of the matching equations.

Let us finally note that most approaches \cite{lavretsky2009combined,chowdhary2010concurrent,chowdhary2013concurrent,cho2018composite,roy2018combined,lee2019concurrent} in the multi-variable MRAC literature require $ B_{\rm s} $ to have full column rank, which in turns guarantees a unique solution $ (K,L) $ of the matching equations \eqref{mc} \cite{tao2014multivariable}.
This is not required in the proposed MRAC method, which remains valid even when an infinite number of solutions of the matching equations \eqref{mc} exist. 
This will be illustrated in the numerical example in the next section.
%

\section{Simulation Results}

In this section, the proposed MRAC method is demonstrated on a numerical system and on a practical example involving highly maneuverable aircraft dynamics.
To validate the method under diverse data excitation conditions, we consider, for each system, four simulation scenarios: (i) two simulations with reference input signals drawn from the normal distribution; (ii) two simulations with constant reference input signals.
Using four random initial conditions, convergence to four different controllers is observed. 
This is expected as the matching equations have an infinite number of solutions.
For each simulation, we set $ \gamma=1.99 $ and $ \sigma=10^2 $.
To terminate the simulation when sufficient convergence is achieved, we set, motivated by \eqref{convergence}, the stopping criterion
\begin{equation}\label{epsilon}
\left\|\Phi_X(t)\Theta(t)-
\begin{bmatrix}
I&0\\A_{\rm m}&B_{\rm m}
\end{bmatrix}\right\|_F^2\leq\varepsilon,
\end{equation}
with $ \varepsilon = 10^{-10} $. 
During the implementation, the informative time $T^*$ is initialized to $ T^*=\infty $.
At every time step $ t $, we check whether the data $ (U_-(t),X(t)) $ satisfy \eqref{image}, and update $ T^* $ accordingly.
In all simulations, we verify that \mbox{$T^* \leq n+m$}, consistent with Lemma \ref{LT}, and that the control gains $(\hat{K}(t),\hat{L}(t))$ in \eqref{gain} converge to a solution of the matching equations \eqref{mc}, as guaranteed by Theorem \ref{TMRAC}.

\subsection{Numerical system}

Consider the controllable system with matrices.
\begin{align*}
A_{\rm s} &=
\begin{bmatrix}
-1.1   & -0.85  & -0.1   & -0.62  \\
-0.24  & -0.65  &  0.77  & -0.34  \\
-1.57  & -0.26  & -0.36  & -1.2   \\
0.33  & -0.99  & -0.43  &  0.56
\end{bmatrix}\!,\\
B_{\rm s} &=
\begin{bmatrix}
 0.35 &  0.52 &  0.01 \\
 0.39 & -0.14 &  1.45 \\
 1.04 &  0.98 &  1.16 \\
 0.13 &  0.34 & -0.29
\end{bmatrix}\!,
\end{align*}
being unknown for control design. 
Further consider the reference model with matrices
\begin{align*}
A_{\rm m} &=
\begin{bmatrix}
-0.75 & -0.32 &  0.24 & -0.27 \\
0.15 &  0.66 & -0.29 &  0.05 \\
-0.53 &  1.88 & -0.48 & -0.16 \\
0.46 & -0.94 & -0.01 &  0.69
\end{bmatrix}\!,\\
B_{\rm m} &=
\begin{bmatrix}
0.52 & -0.86 & -0.69 \\
-0.14 &  1.2  &  0.67 \\
0.98 & -0.86 & -0.92 \\
0.34 & -0.76 & -0.55
\end{bmatrix}\!.
\end{align*}

\begin{figure*}[htbp]
	\centering
	\subfigure[State tracking errors in $ \mathcal{S}_1 $]
	{
		\includegraphics[trim=90bp 266bp 100bp 280bp, clip, height=4.26cm]{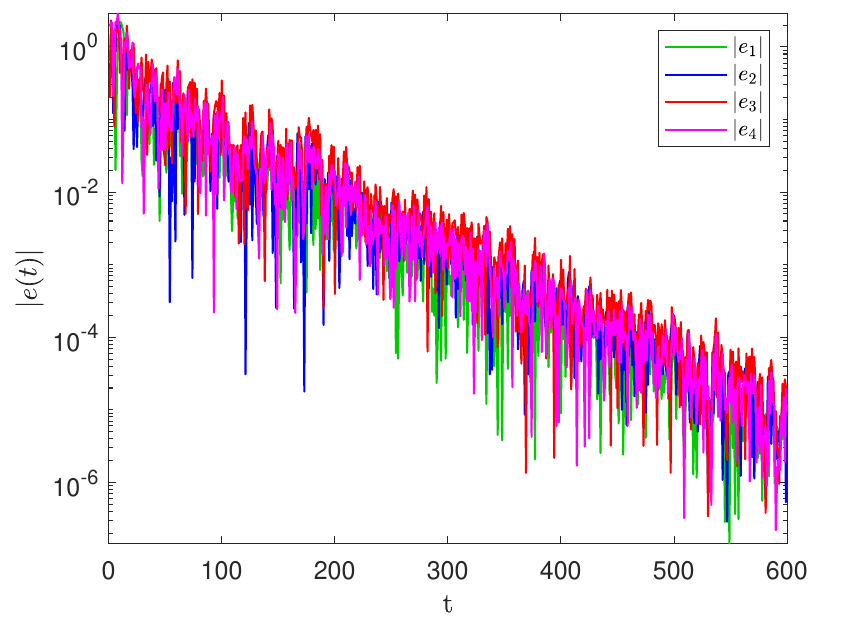}}
	\subfigure[Norm of matching error in $ \mathcal{S}_1 $]
	{
		\includegraphics[trim=90bp 266bp 100bp 280bp, clip, height=4.26cm]{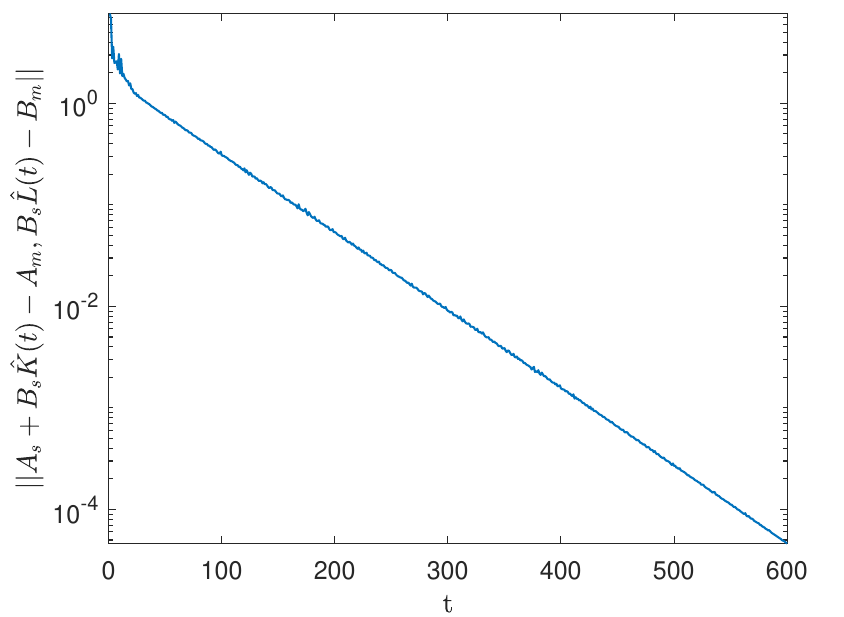}}
	\subfigure[Input signals in $ \mathcal{S}_1 $]
	{\includegraphics[trim=90bp 266bp 100bp 280bp, clip, height=4.26cm]{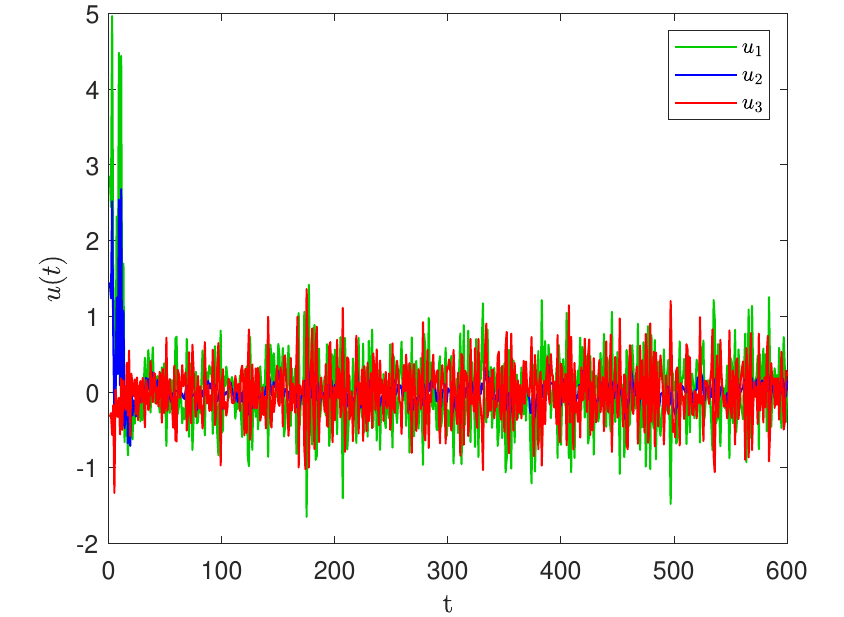}
		\label{rinput1}}
	\subfigure[State tracking errors in $ \mathcal{S}_2 $]
	{
		\includegraphics[trim=90bp 266bp 100bp 280bp, clip, height=4.26cm]{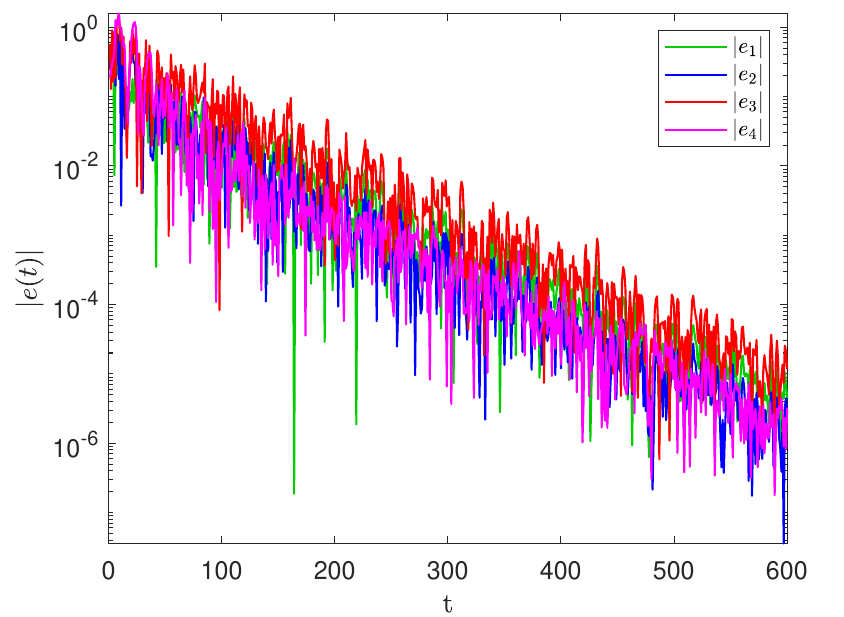}}
	\subfigure[Norm of matching error in $ \mathcal{S}_2 $]
	{
		\includegraphics[trim=90bp 266bp 100bp 280bp, clip, height=4.26cm]{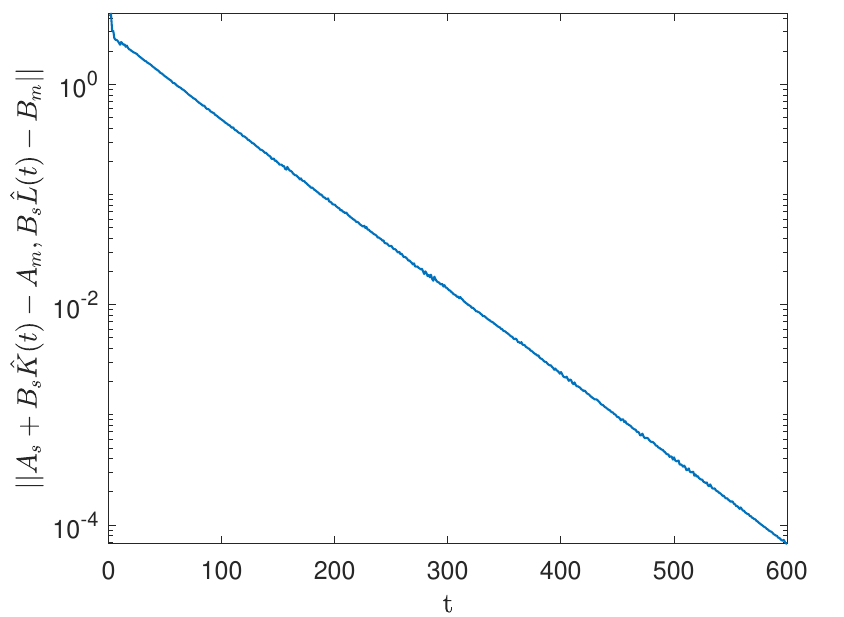}}
	\subfigure[Input signals in $ \mathcal{S}_2 $]
	{\includegraphics[trim=90bp 266bp 100bp 280bp, clip, height=4.26cm]{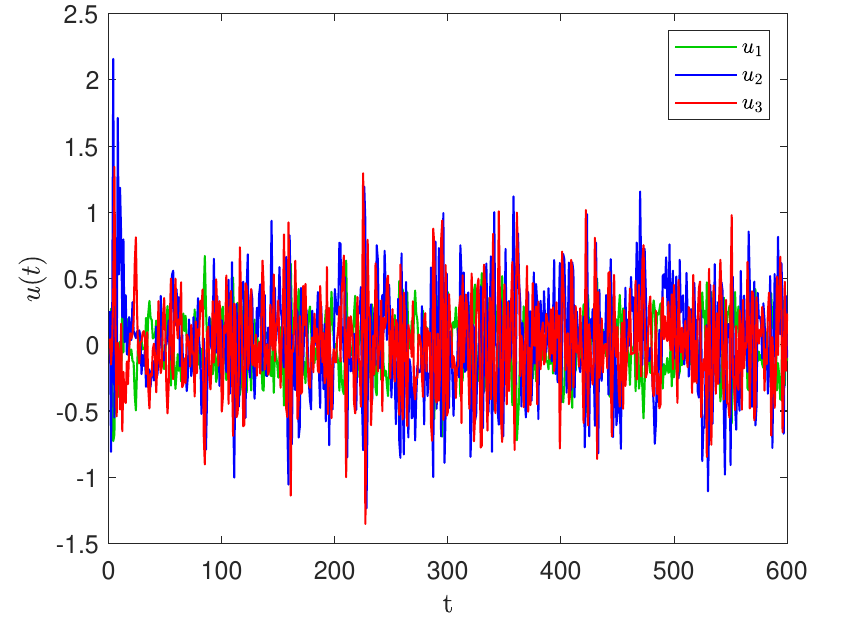}
		\label{rinput2}}
	\caption{
		State tracking errors, norms of the matching error and input signals in $ \mathcal{S}_1 $ and $ \mathcal{S}_2 $.
	}
	\label{rerror}
\end{figure*}

\begin{figure*}[!b]
	\hrulefill
	\begin{align}
	\label{rS1}
	X(6) &= \begin{bmatrix}
	0.4515 &  0.5915 & -0.5977 &  1.4086 & -0.4482 & -0.7842 &  0.9521 \\
	-1.1013 &  0.3756 & -0.5818 &  0.4610 &  0.9848 &  0.0235 & -1.6123 \\
	-0.1247 &  0.1476 &  0.2960 &  2.0724 &  1.3014 & -0.2891 &  0.2497 \\
	0.9257 &  2.2021 &  1.9527 &  2.1648 &  1.9960 & -0.0567 &  0.1607
	\end{bmatrix},
	\\
	\label{rS2}
	X(6) &= \begin{bmatrix}
	-0.8873 &  0.4737 & -0.2669 & -0.2928 &  0.8949 & -0.6439 & -0.2982 \\
	-0.1354 & -0.1576 &  0.1929 & -0.5193 &  0.1309 &  0.2239 &  1.1206 \\
	-0.4604 &  0.2365 & -0.4934 & -0.4175 &  0.9198 & -0.5817 &  0.7704 \\
	0.3136 & -0.0666 &  0.2349 & -0.1310 &  0.5287 &  0.5892 & -0.5182
	\end{bmatrix},
	\end{align}
\end{figure*}

Two simulations $ \mathcal{S}_1 $ and $ \mathcal{S}_2 $ with reference input signals drawn from the normal distribution give the following results.
The simulation $ \mathcal{S}_1 $ stops at $ t=683 $, providing the final control gains (here and in the following we truncate numbers to four digits after the decimal point).
\begin{align*}
	\hat{K}&=\begin{bmatrix}
		0.4429 &  0.3649 &  1.0418 &  0.4429 \\
		0.3714 &  0.7567 & -0.0279 &  0.3714 \\
		0.1857 &  0.8784 & -1.0139 &  0.1857
	\end{bmatrix}\!,\\
	\hat{L}&=\begin{bmatrix}
		0.9638 & -2.0056 & -1.4847 \\
		0.3575 & -0.3296 & -0.3435 \\
		-0.3213 &  1.3352 &  0.8282
	\end{bmatrix}\!.
\end{align*}
If we use the (unavailable) knowledge of $(A_{\rm s},B_{\rm s})$ one can verify that such gains correspond to a matching error\footnote{The matching error is calculated before rounding to four decimals.} 
$$||A_{\rm s}+B_{\rm s}\hat{K}-A_{\rm m},B_{\rm s}\hat{L}-B_{\rm m}||=1.0897\times10^{-5}.$$
The simulation $ \mathcal{S}_2 $  stops at $ t=686 $, providing the final gains
\begin{align*}
	\hat{K}&=\begin{bmatrix}
		-0.5377 & -1.2656 &  0.3802 & -0.5377 \\
		1.0252 &  1.8437 &  0.4132 &  1.0252 \\
		0.5126 &  1.4219 & -0.7934 &  0.5126
	\end{bmatrix}\!,\\
	\hat{L}&=\begin{bmatrix}
		-0.3477 & -0.0325 &  0.1576 \\
		1.2318 & -1.6450 & -1.4384 \\
		0.1159 &  0.6775 &  0.2808
	\end{bmatrix}
\end{align*}
with matching error $1.5017\times10^{-5}$.
For $ \mathcal{S}_1 $ and $ \mathcal{S}_2 $, we obtain $ X(6) $ in \eqref{rS1} and \eqref{rS2}, respectively.
Figure \ref{rerror} shows the state tracking errors, the norms of the matching error and the input signals for $ \mathcal{S}_1 $ and $ \mathcal{S}_2 $.

\begin{figure*}[htbp]
	\centering
	\subfigure[State tracking errors in $ \mathcal{S}_3 $]
	{\includegraphics[trim=90bp 266bp 100bp 280bp, clip, height=4.26cm]{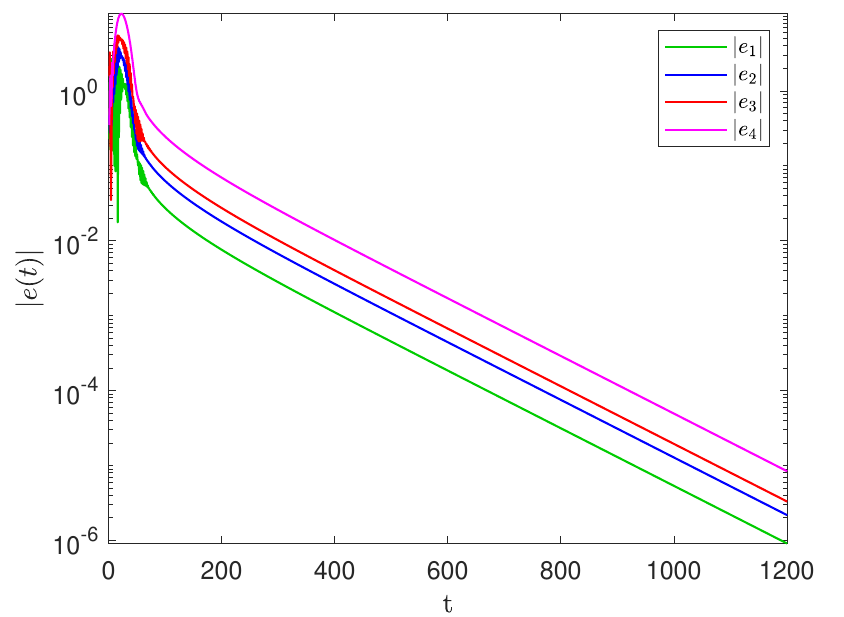}}
	\subfigure[Norm of matching error in $ \mathcal{S}_3 $]
	{\includegraphics[trim=90bp 266bp 100bp 280bp, clip, height=4.26cm]{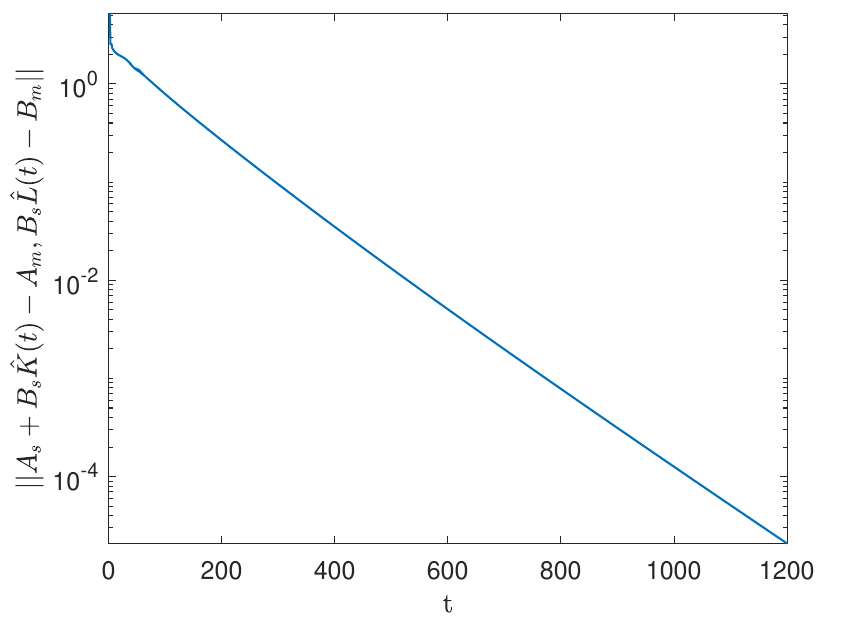}}
	\subfigure[Input signals in $ \mathcal{S}_3 $]
	{\includegraphics[trim=90bp 266bp 100bp 280bp, clip, height=4.26cm]{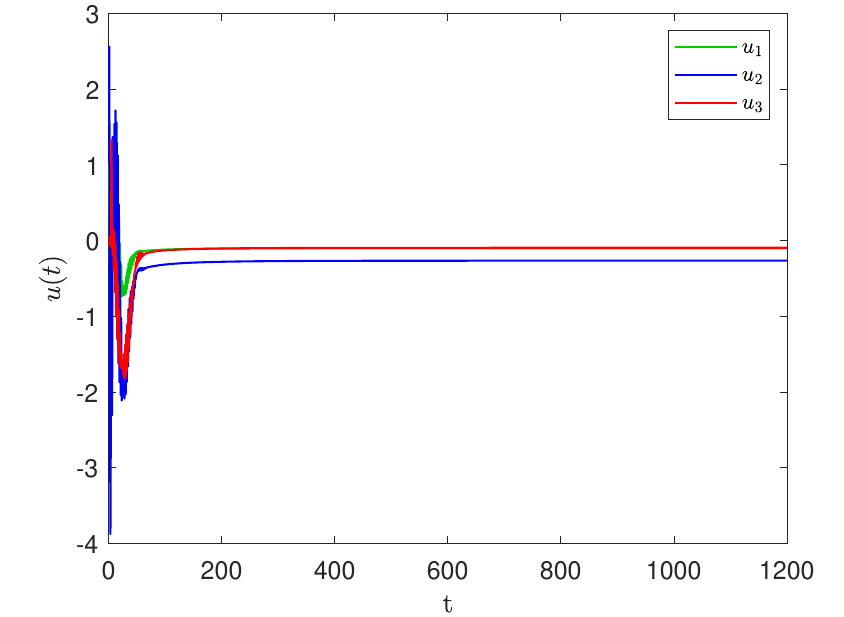}
		\label{rinput3}}
	\subfigure[State tracking errors in $ \mathcal{S}_4 $]
	{\includegraphics[trim=90bp 266bp 100bp 280bp, clip, height=4.26cm]{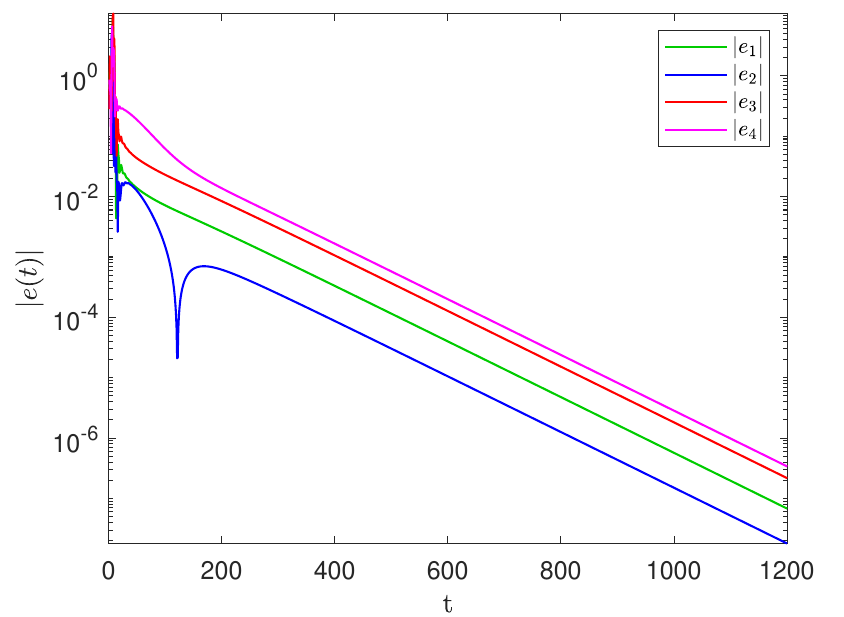}}
	\subfigure[Norm of matching error in $ \mathcal{S}_4 $]
	{\includegraphics[trim=90bp 266bp 100bp 280bp, clip, height=4.26cm]{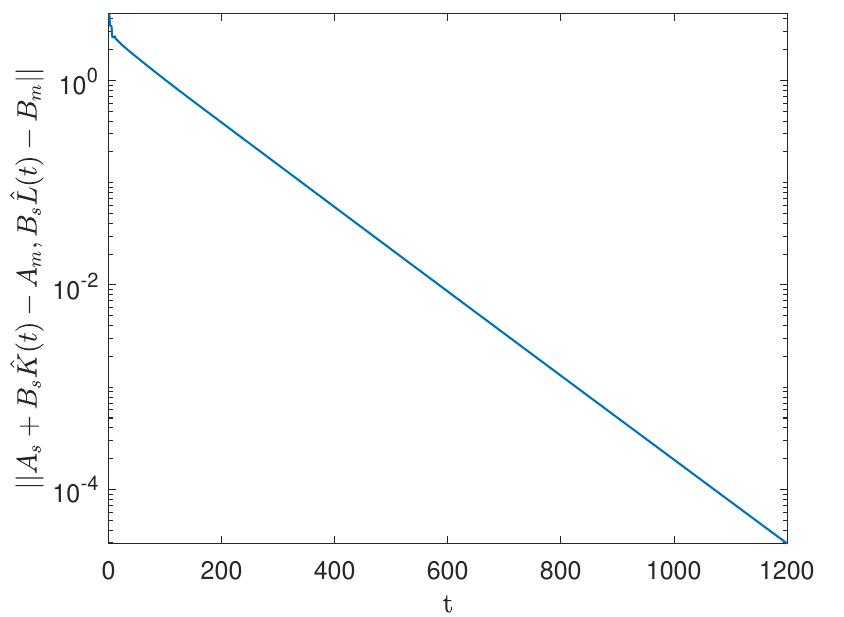}}
	\subfigure[Input signals in $ \mathcal{S}_4 $]
	{\includegraphics[trim=90bp 266bp 100bp 280bp, clip, height=4.26cm]{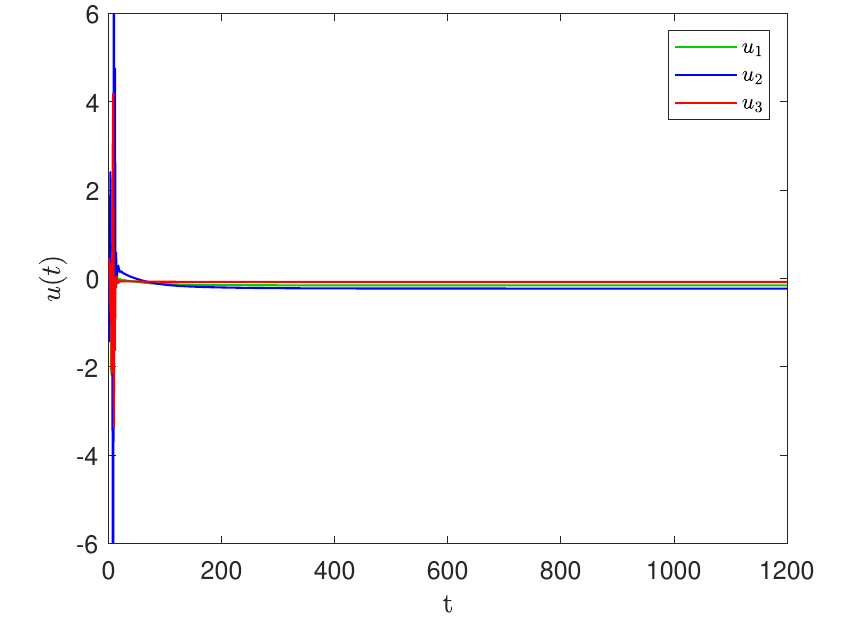}
		\label{rinput4}}
	\caption{
		State tracking errors, norms of the matching error and input signals in $ \mathcal{S}_3 $ and $ \mathcal{S}_4 $.
	}
	\label{rerrorc}
\end{figure*}

\begin{figure*}[!b]
	\hrulefill
	\begin{align}
	\label{rS3}
	X(6) &= \begin{bmatrix}
	-0.8962 &  0.8273 &  0.1253 & -1.7994 &  0.2101 &  1.5451 & -2.6355 \\
	0.2298 &  0.4897 &  0.6467 &  0.5691 & -1.8402 &  1.9382 &  0.5407 \\
	0.9684 &  1.4334 &  1.5077 & -2.9396 &  0.1314 &  0.2842 &  0.0059 \\
	1.1103 &  0.1631 &  0.2554 & -1.5738 & -2.2694 & -0.1277 & -1.5966
	\end{bmatrix},
	\\
	\label{rS4}
	X(6) &= \begin{bmatrix}
	0.1246 & -0.3486 & -0.7091 &  1.1875 &  0.4900 & -1.6225 &  2.1005 \\
	-0.3700 &  0.4088 &  0.4716 & -0.3069 &  0.6262 & -0.1266 & -4.5740 \\
	-1.1252 & -0.2829 & -0.3612 &  2.4710 & -0.2162 & -1.5610 & -0.0264 \\
	-0.2358 &  0.0883 & -0.9631 & -0.9351 &  0.1379 & -0.7048 &  0.4694
	\end{bmatrix},
	\end{align}
\end{figure*}

Two simulations $ \mathcal{S}_3 $ and $ \mathcal{S}_4 $ with constant reference input $ r(t)=[0.1\ 0.1\ 0.1\ 0.1]^{\top}\ \forall t\in \mathbb{Z}_+ $ give the following results.
The simulation $ \mathcal{S}_3 $ stops at $ t=1260 $, providing the final control gains
\begin{align*}
	\hat{K}&=\begin{bmatrix}
		0.1969 &  0.2704 &  0.2468 &  0.1969 \\
		0.5354 &  0.8197 &  0.5021 &  0.5354 \\
		0.2677 &  0.9099 & -0.7489 &  0.2677
	\end{bmatrix}\!,\\
	\hat{L}&=\begin{bmatrix}
		0.3203 & -0.5672 & -0.4438 \\
		0.7864 & -1.2885 & -1.0375 \\
		-0.1068 &  0.8557 &  0.4813
	\end{bmatrix}
\end{align*}
with matching error $1.2286\times10^{-5}$.
The simulation $ \mathcal{S}_4 $ stops at $ t=1265 $, providing the final gains
\begin{align*}
	\hat{K}&=\begin{bmatrix}
		0.3694 &  0.9106 & -0.3437 &  0.3694 \\
		0.4204 &  0.3929 &  0.8958 &  0.4204 \\
		0.2102 &  0.6965 & -0.5521 &  0.2102
	\end{bmatrix}\!,\\
	\hat{L}&=\begin{bmatrix}
		0.1975 &  0.1462 & -0.0257 \\
		0.8683 & -1.7641 & -1.3162 \\
		-0.0658 &  0.6179 &  0.3419
	\end{bmatrix}
\end{align*}
with matching error $1.6208\times10^{-5}$.
For $ \mathcal{S}_3 $ and $ \mathcal{S}_4 $, we obtain $ X(6) $ in \eqref{rS3} and \eqref{rS4}, respectively.
Figure \ref{rerrorc} shows the state tracking errors, the norms of the matching error and the input signals for $ \mathcal{S}_3 $ and $ \mathcal{S}_4 $.

Despite the different level of data excitation that is evident from the plots of Figure \ref{rerror} and Figure \ref{rerrorc}, we verify that the informative time is always $ T^*=6 $. 
This implies that, despite diverse data excitation conditions, $ 6 $ data samples are always sufficient to achieve data	informativity for model	reference control.

\subsection{Highly maneuverable aircraft example}

We consider the unstable longitudinal dynamics of a highly maneuverable aircraft \cite{hartmann1979control,yuan2018robust} with three states (angle of attack, pitch rate, and pitch angle) and four inputs (elevator, elevon, canard, and symmetric aileron). 
For better numerical conditioning, we scale the angle of attack by 100. 
With sampling time of 0.01s, we get a system model in the form \eqref{s}, with
\vspace*{0.6ex}
\begin{align*}
A_{\rm s} &=
\begin{bmatrix}
0.9810 & 0.9831 & -0.0007 \\
0.0012 & 0.9737 & 0 \\
0 & 0.01 & 1 \\
\end{bmatrix}\!,\\
B_{\rm s} &=
\begin{bmatrix}
-0.2436 & -0.1708 & -0.0050 & -0.1997 \\
-0.4621 & -0.3160 & 0.2240 & -0.3118 \\
0 & 0 & 0 & 0 \\
\end{bmatrix}\!,
\end{align*}
and a reference model in the form \eqref{r}, with
\vspace*{0.6ex}
\begin{equation}\label{ambm}
A_{\rm m} = \begin{bmatrix}
0.9800 & 0.6484 & -0.7487 \\
-0.0008 & 0.2964 & -1.5178 \\
0 & 0.01 & 1 \\
\end{bmatrix}\!,\ B_{\rm m}=B_{\rm s}.
\end{equation}
\\
The choice $ B_{\rm m}=B_{\rm s} $ is based on the setting in \cite{hartmann1979control,yuan2018robust}. 
Nevertheless, we stress on the fact that $ B_{\rm s} $ is unknown: hence, even though we have $ B_{\rm m}=B_{\rm s} $ in the simulation, we do not know that $ B_{\rm m}=B_{\rm s} $ for control design.

\begin{figure*}[htbp]
	\centering
	\subfigure[State tracking errors in $ \mathcal{S}_5 $]
	{
		\includegraphics[trim=90bp 266bp 100bp 280bp, clip, height=4.26cm]{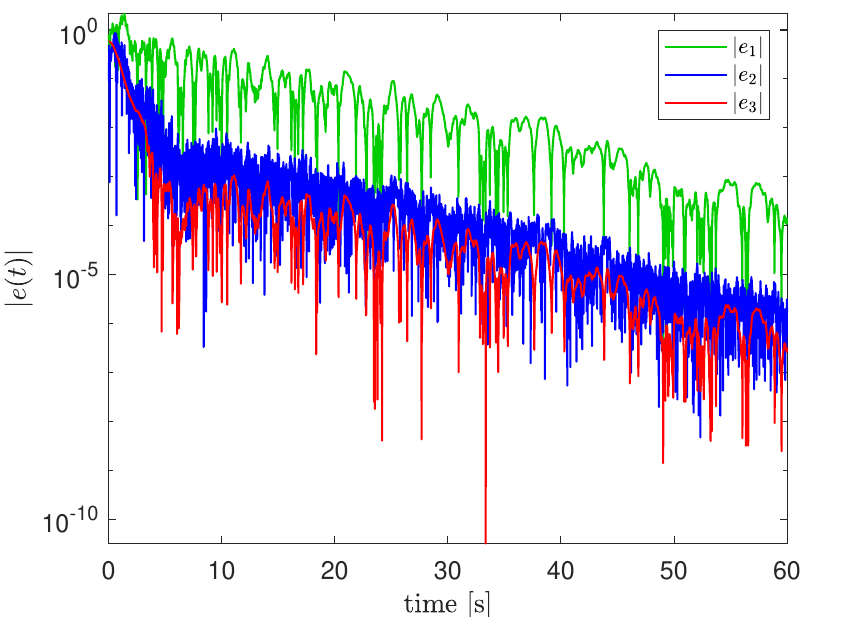}}
	\subfigure[Norm of matching error in $ \mathcal{S}_5 $]
	{
		\includegraphics[trim=90bp 266bp 100bp 280bp, clip, height=4.26cm]{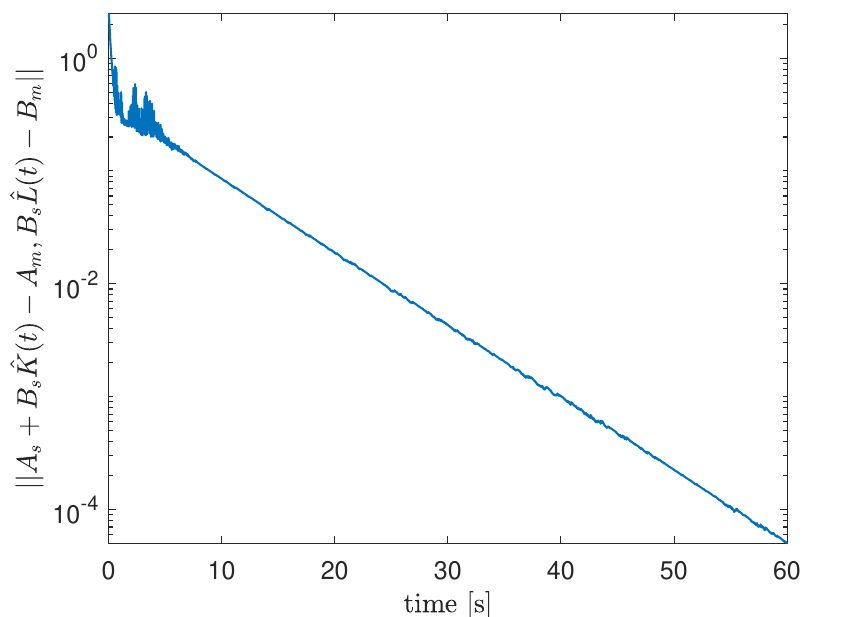}}
	\subfigure[Input signals in $ \mathcal{S}_5 $]
	{\includegraphics[trim=90bp 266bp 100bp 280bp, clip,, height=4.26cm]{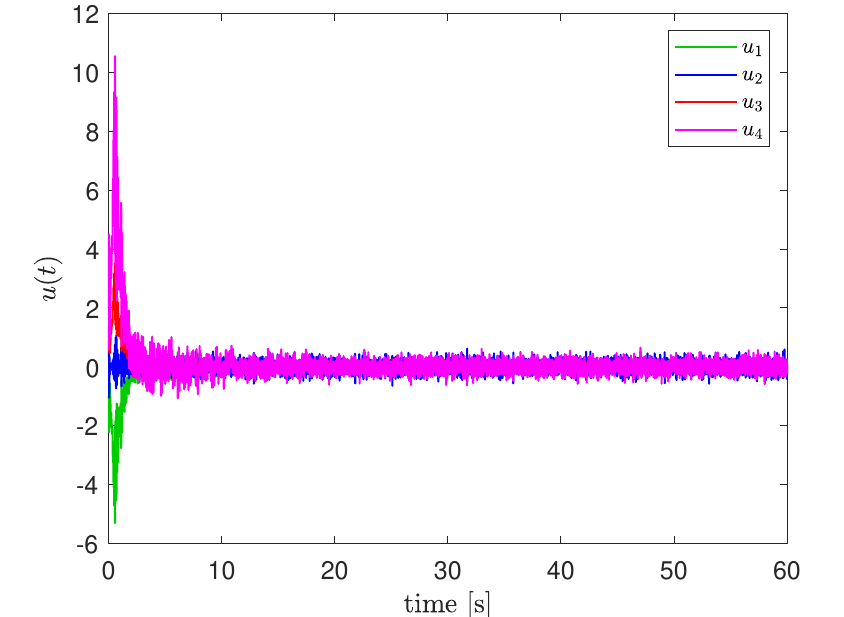}
		\label{input1}}
	\subfigure[State tracking errors in $ \mathcal{S}_6 $]
	{
		\includegraphics[trim=90bp 266bp 100bp 280bp, clip, height=4.26cm]{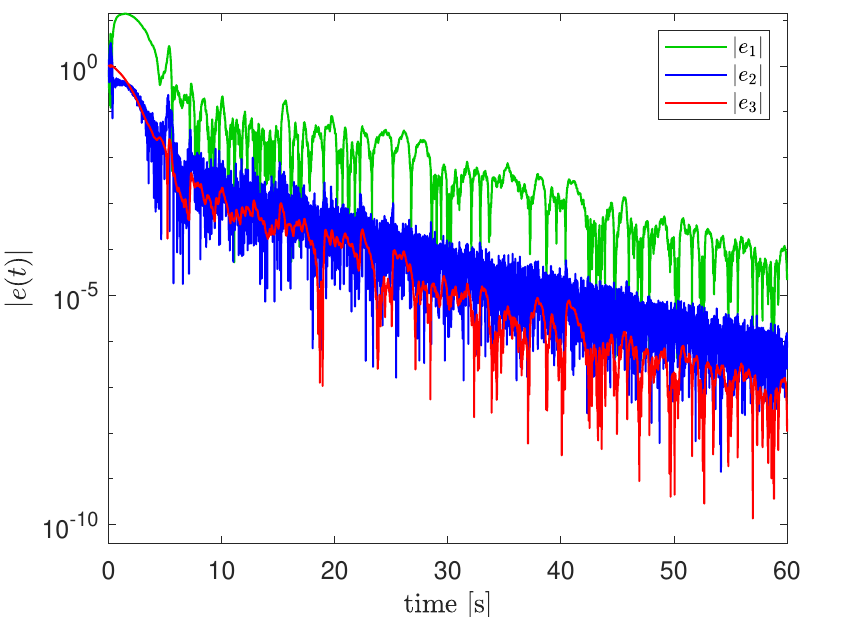}}
	\subfigure[Norm of matching error in $ \mathcal{S}_6 $]
	{
		\includegraphics[trim=90bp 266bp 100bp 280bp, clip, height=4.26cm]{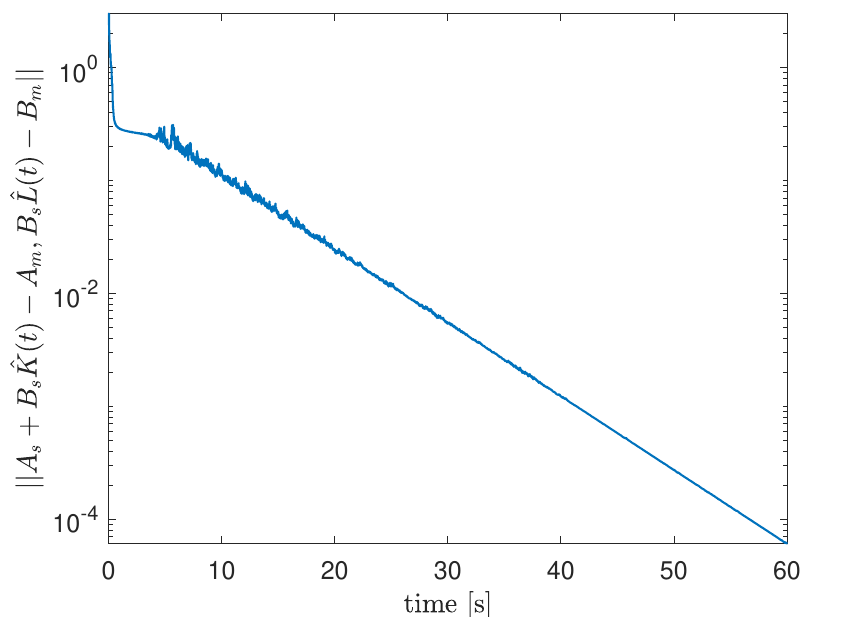}}
	\subfigure[Input signals in $ \mathcal{S}_6 $]
	{\includegraphics[trim=90bp 266bp 100bp 280bp, clip, height=4.26cm]{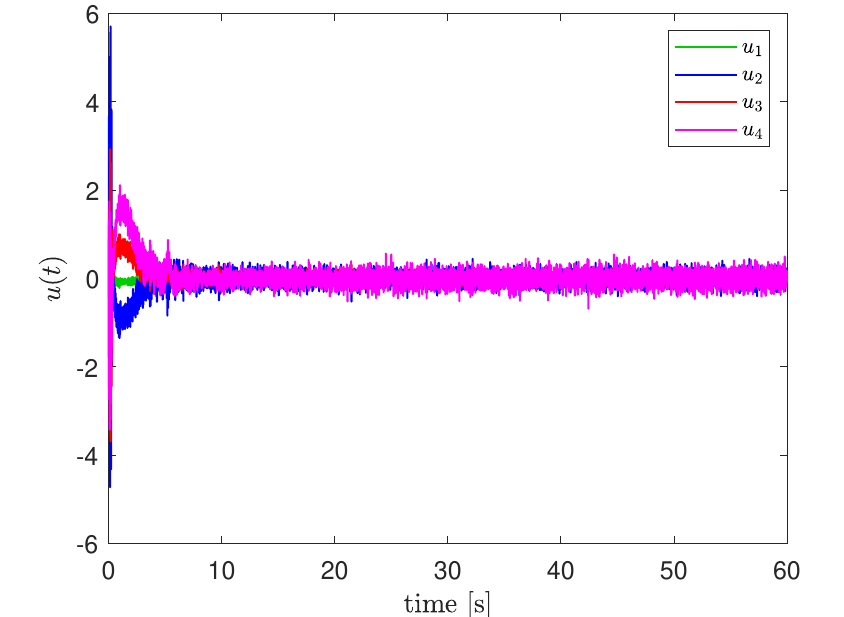}
		\label{input2}}
	\caption{State tracking errors, norms of the matching error and input signals in $ \mathcal{S}_5 $ and $ \mathcal{S}_6 $.}
	\label{error}
\end{figure*}

Two simulations $ \mathcal{S}_5 $ and $ \mathcal{S}_6 $ with reference input signals drawn from the normal distribution give the following results.
The simulation $ \mathcal{S}_5 $ stops at $ 68.39 $s, providing the final control gains
\vspace*{0.6ex}
\begin{align*}
	\hat{K}&=\begin{bmatrix}
	-0.0006 & -0.0784 & -0.1616 \\
	0.0041 &  1.4191 &  3.1798 \\
	-0.0009 & -0.3930 & -0.8904 \\
	0.0023 &  0.5680 &  1.2454
	\end{bmatrix}\!,
	\\
	\hat{L}&=\begin{bmatrix}
	-0.1516 & -0.1309 & -0.7200 & -0.3301 \\
	0.9708 &  0.6646 & -0.4491 &  0.6612 \\
	-0.1996 & -0.1175 &  0.6504 &  0.0243 \\
	0.5795 &  0.4495 &  1.2712 &  0.8365
	\end{bmatrix}\!,
\end{align*}
with matching error
$$||A_{\rm s}+B_{\rm s}\hat{K}-A_{\rm m},B_{\rm s}\hat{L}-B_{\rm m}||=1.4443\times10^{-5}.$$
Such matching error was below $ 10^{-3} $ at $ 40.05 $s and below $ 10^{-4} $ at $ 55.24 $s.
The simulation $ \mathcal{S}_6 $ stops at $ 68.64 $s, providing the final gains
\begin{align*}
	\hat{K}&=\begin{bmatrix}
	0.0010 &  0.3706 &  0.8318 \\
	0.0019 &  0.7646 &  1.7275 \\
	-0.0008 & -0.3740 & -0.8491 \\
	0.0022 &  0.5795 &  1.2746
	\end{bmatrix}\!,\\
	\hat{L}&=\begin{bmatrix}
	0.2440 &  0.1644 & -0.1896 &  0.1442 \\
	0.4223 &  0.2613 & -1.0074 &  0.0546 \\
	-0.1768 & -0.0992 &  0.7188 &  0.0624 \\
	0.5654 &  0.4337 &  1.0998 &  0.7757
	\end{bmatrix}
\end{align*}
with matching error $1.6612\times10^{-5}$.
Such error was below $ 10^{-3} $ at $ 41.38 $s and below $ 10^{-4} $ at $ 56.66 $s.
We also verify that the informative time is $ T^*=5 $ in both $ \mathcal{S}_5 $ and $ \mathcal{S}_6 $, which implies that $ 5 $ data samples are	sufficient to achieve	informativity for model	reference	control.
In particular, in $ \mathcal{S}_5 $, we obtain the following set of state data
{
	\setlength{\arraycolsep}{2pt}
	\begin{align*}
	&X(5) = \\
	&\!\begin{bmatrix}
	0.3200 &  0.6376 &  0.8311 &  0.6256 &  0.4428 &  0.5105 \\
	0.5871 &  0.4239 &  0.2900 &  0.0060 & -0.0968 &  0.4126 \\
	-0.5974 & -0.5915 & -0.5873 & -0.5844 & -0.5843 &  -0.5853
	\end{bmatrix}\!,
	\end{align*}
	while in $ \mathcal{S}_6 $, we obtain
	\vspace*{0.6ex}
	\begin{align*}
	&X(5)= \\
	&\!\begin{bmatrix}
	0.7030 &  1.2024 &  0.8685 & -0.2838 & -0.9618 & -1.0249 \\
	0.6777 & -0.0262 & -1.2978 & -0.6420 & -0.8612 & 0.4623 \\
	-1.0029 & -0.9961 & -0.9964 & -1.0093 & -1.0158 & -1.0244
	\end{bmatrix}\!.
	\end{align*}\\
}Figure \ref{error} shows the state tracking errors, the norms of the matching error and the input signals for $ \mathcal{S}_5 $ and $ \mathcal{S}_6 $.

\begin{figure*}[htbp]
	\centering
	\subfigure[State tracking errors in $ \mathcal{S}_7 $]
	{\includegraphics[trim=90bp 266bp 100bp 280bp, clip, height=4.26cm]{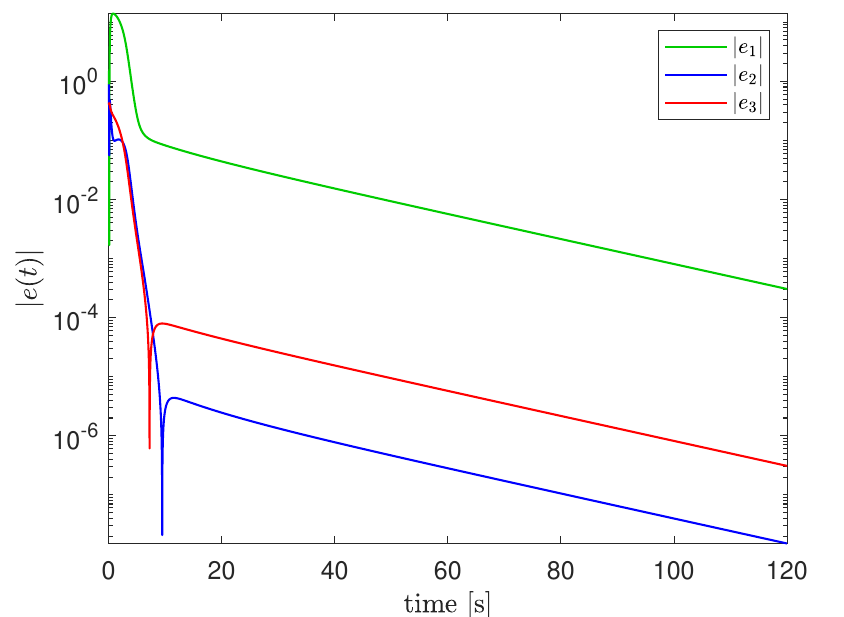}}
	\subfigure[Norm of matching error in $ \mathcal{S}_7 $]
	{\includegraphics[trim=90bp 266bp 100bp 280bp, clip, height=4.26cm]{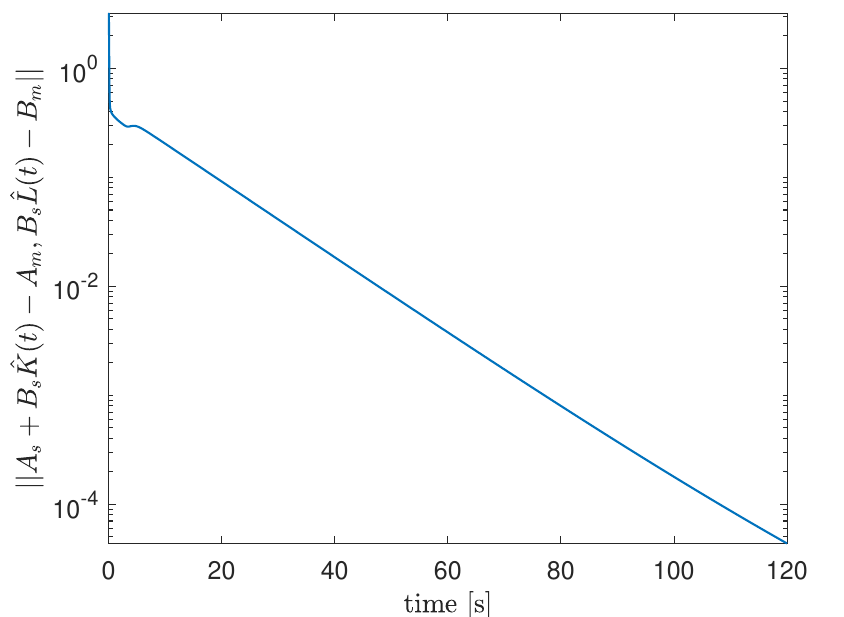}}
	\subfigure[Input signals in $ \mathcal{S}_7 $]
	{\includegraphics[trim=90bp 266bp 100bp 280bp, clip, height=4.26cm]{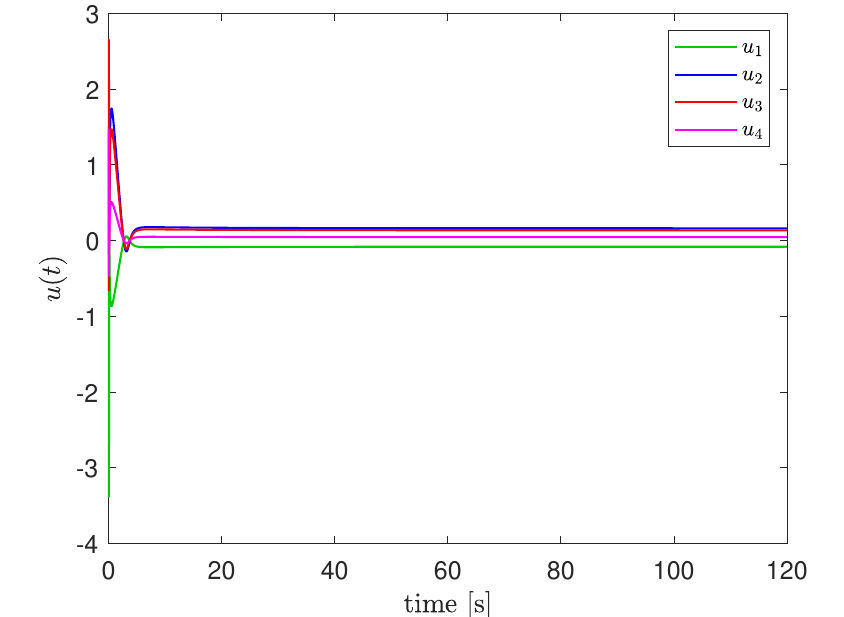}
		\label{input3}}
	\subfigure[State tracking errors in $ \mathcal{S}_8 $]
	{\includegraphics[trim=90bp 266bp 100bp 280bp, clip, height=4.26cm]{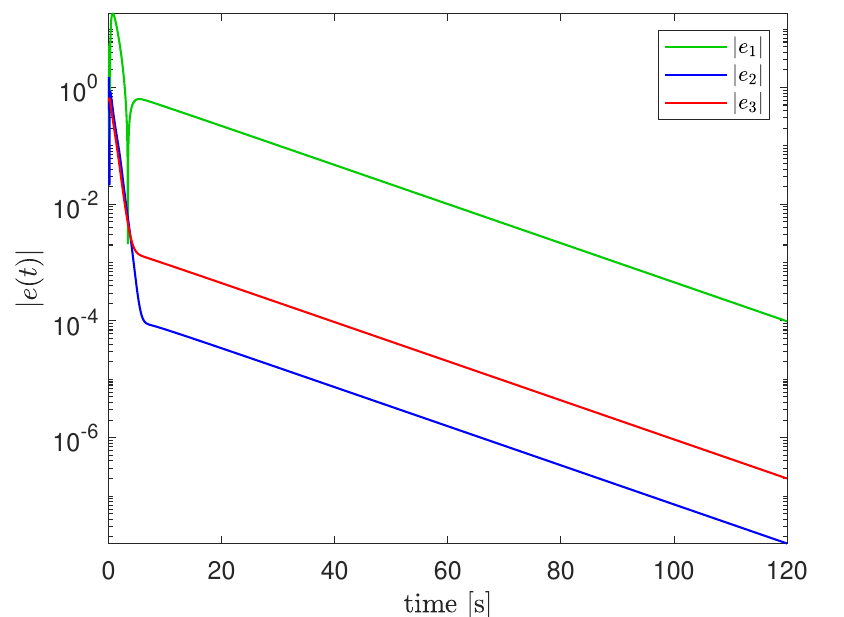}}
	\subfigure[Norm of matching error in $ \mathcal{S}_8 $]
	{\includegraphics[trim=90bp 266bp 100bp 280bp, clip, height=4.26cm]{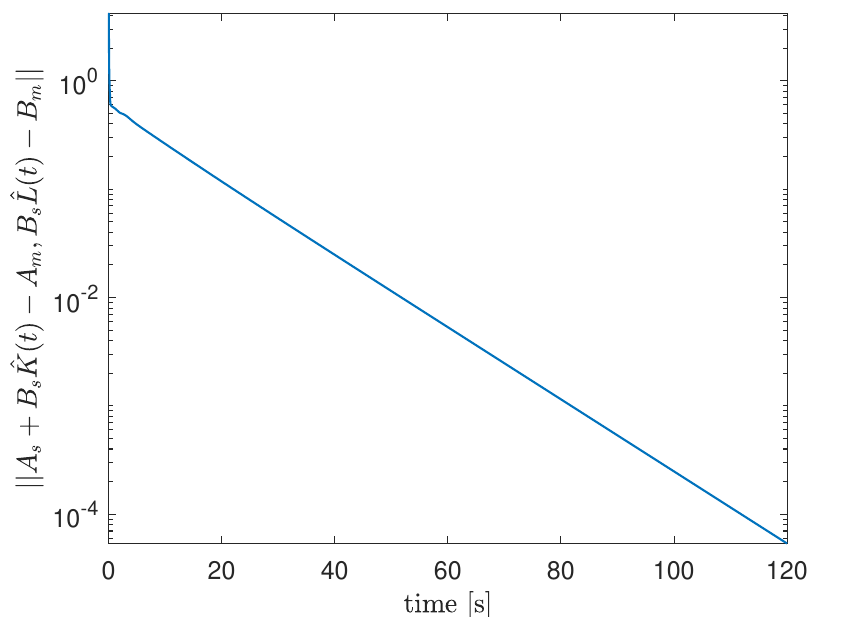}}
	\subfigure[Input signals in $ \mathcal{S}_8 $]
	{\includegraphics[trim=90bp 266bp 100bp 280bp, clip, height=4.26cm]{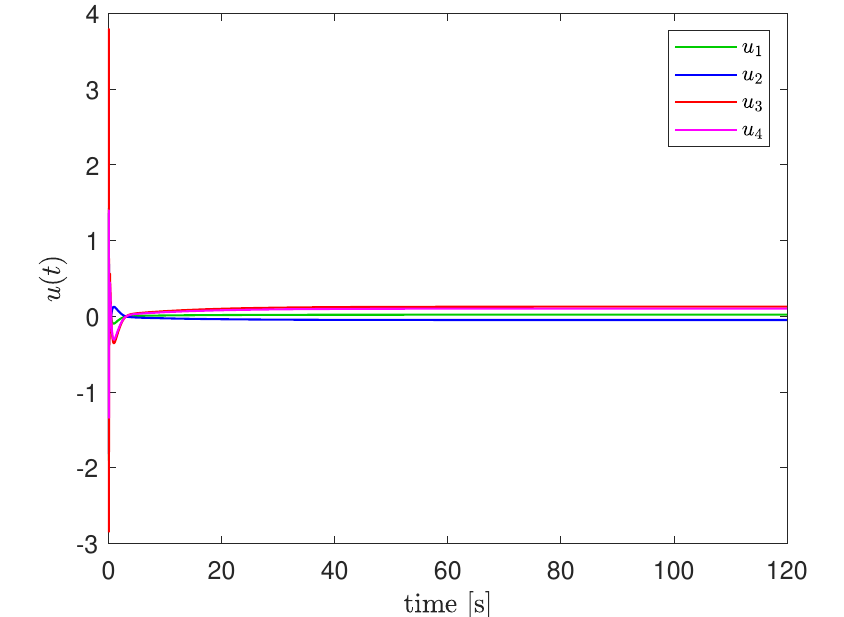}
		\label{input4}}
	\caption{State tracking errors, norms of the matching error and input signals in $ \mathcal{S}_7 $ and $ \mathcal{S}_8 $.}
	\label{errorc}
\end{figure*}

The two simulations $ \mathcal{S}_7 $ and $ \mathcal{S}_8 $ with constant reference input $ r(t)=[0.1\ 0.1\ 0.1\ 0.1]^{\top}\ \forall t\in \mathbb{Z}_+ $ give the following results.
The simulation $ \mathcal{S}_7 $ stops at $ 135.02 $s, providing the final control gains
\vspace*{0.6ex}
\begin{align*}
	\hat{K}&=\begin{bmatrix}
	0.0025 &  0.9383 &  2.1114 \\
	0.0032 &  0.9843 &  2.1872 \\
	-0.0001 & -0.1270 & -0.3000 \\
	-0.0008 & -0.3070 & -0.6931
	\end{bmatrix}\!,\\
	\hat{L}&=\begin{bmatrix}
	0.5782 &  0.3782 & -0.7808 &  0.2464 \\
	0.8026 &  0.5852 &  0.6704 &  0.8457 \\
	0.0220 &  0.0451 &  0.8672 &  0.2668 \\
	-0.1724 & -0.1077 &  0.3824 & -0.0306
	\end{bmatrix}
\end{align*}\\
with matching error $1.6288\times10^{-5}$.
Such matching error was below $ 10^{-3} $ at $ 77.02 $s and below $ 10^{-4} $ at $ 108.01 $s.
The simulation $ \mathcal{S}_8 $ stops at $ 135.03 $s, providing the final gains
\vspace*{0.6ex}
\begin{align*}
	\hat{K}&=\begin{bmatrix}
	0.0026 &  0.8951 &  2.0031 \\
	0.0006 &  0.2586 &  0.5848 \\
	-0.0005 & -0.2967 & -0.6792 \\
	0.0013 &  0.3706 &  0.8189
	\end{bmatrix}\!,
	\\
	\hat{L}&=\begin{bmatrix}
	 0.6298 &  0.4360 & -0.1505 &  0.4694 \\
	0.1389 &  0.0846 & -0.3706 &  0.0067 \\
	-0.1010 & -0.0410 &  0.8680 &  0.1669 \\
	0.3353 &  0.2521 &  0.5038 &  0.4175
	\end{bmatrix}
\end{align*}
with matching error $1.6992\times10^{-5}$.
Such error was below $ 10^{-3} $ at $ 81.81 $s and below $ 10^{-4} $ at $ 111.88 $s.
As in $ \mathcal{S}_5 $ and $ \mathcal{S}_6 $, the informative time is $ T^*=5 $ in both $ \mathcal{S}_7 $ and $ \mathcal{S}_8 $.
In particular, in $ \mathcal{S}_7 $, we obtain
{
	\setlength{\arraycolsep}{2pt}
	\begin{align*}
	&X(5)= \\
	&\!\begin{bmatrix}
	0.5869 & -0.4179 & -1.1217 & -1.6710 & -1.1897 & -1.3823 \\
	-1.1421 & -0.6244 & -0.9753 & 0.5102 & 0.3004 & -0.7382 \\
	0.4350 & 0.4236 & 0.4174 & 0.4076 & 0.4127 & 0.4157
	\end{bmatrix}\!,\qquad
	\end{align*}
	while in $ \mathcal{S}_8 $, we obtain
	\begin{align*}
	&X(5)= \\
	&\!\begin{bmatrix}
	2.0434 & 1.3784 & 0.7991 & -0.3101 & 0.0711 & 0.8231 \\
	-0.7812 & -0.1990 & -1.5956 & 0.3143 & 0.5676 & 0.4209 \\
	0.6150 & 0.6072 & 0.6052 & 0.5892 & 0.5924 & 0.5981
	\end{bmatrix}\!.\qquad
	\end{align*}
}Figure \ref{errorc} shows the state tracking errors, the norms of the matching error and the input signals for $ \mathcal{S}_7 $ and $ \mathcal{S}_8 $.

\vspace*{1ex}
It is worth mentioning that in all simulations, no matter if the reference input is drawn from a normal distribution or is constant, the sets of data \emph{never} satisfy the full rank condition \eqref{frc}. 
Specifically, the data collected in $\mathcal{S}_5$--$\mathcal{S}_8$ at the informative time $ T^*=5 $ satisfy
\begin{equation*}
	\text{rank}\left(\begin{bmatrix}
	U_-(5)\\X_-(5)
	\end{bmatrix}\right)=5<7,
\end{equation*} 
implying that the MRAC problem is solved in scenarios where the collected data do \emph{not} allow unique identification of the system.
This further confirms that the method proposed in the current paper imposes the least restrictive conditions on data collection compared to existing MRAC methods.

	

\section{Conclusions}
	
In this paper, the MRAC problem has been investigated from a novel perspective of data informativity.
Different from previous results on data informativity that used data generated offline, in this paper we have taken an online perspective by verifying data informativity at every time step.
A new MRAC method has been developed that ensures asymptotic convergence of the adaptive gains to a solution of the matching equations, provided such a solution exists.
As a preliminary problem, we have analyzed the relationship between data informativity for (unique) system identification and data informativity for model reference control, demonstrating that the former is only sufficient, but not necessary for informativity for model reference control.
It becomes necessary only in specific scenarios.
Motivated by these findings, we have studied necessary and sufficient conditions for the existence of an adaptive law that guarantees convergence of the adaptive gains to a solution of the matching equations.
Based on this existence analysis, we have devised an input function that switches between an adaptive controller and a term that increases the rank of the data matrix.
Under the assumption that the matching equations have a solution, this input generates informative data for model reference control after a finite number of time steps. 
Moreover, in this situation, the adaptive gains of the controller are shown to converge to a solution of the matching equations.
As compared to state-of-the-art MRAC approaches, the proposed method does not need knowledge of the input matrix and does not require uniqueness of the solution of the matching equations.
A possibility for future research is to explore partial-information scenarios where only input and output data are available.

\bibliographystyle{IEEEtran}
\bibliography{References}

\end{document}